\documentclass[a4paper,11pt]{amsart}
  \usepackage{amsmath}
  \usepackage{url,graphicx}
  \usepackage{amssymb, color, pstricks-add, manfnt}
  \usepackage{amsmath, amsthm,  amsfonts}
  \usepackage{mathrsfs}
  \usepackage{cite}
  \usepackage{enumerate}

  \theoremstyle{plain}
  \newtheorem{Theorem}{Theorem}[section]
  \newtheorem{Lemma}{Lemma}[section]

  \newtheorem{Definition}{Definition}[section]

  \theoremstyle{remark}
  \newtheorem{remark}{Remark}

  \numberwithin{equation}{section}
  \numberwithin{figure}{section}
  \numberwithin{remark}{section}

\usepackage{hyperref}
\usepackage{cite}

\begin{document}

\title{On the solutions to weakly coupled system of $\boldsymbol{k_i}$-Hessian equations}

\author{JingWen Ji}
\address{College of Mathematics and Statistics, Nanjing University of Information Science and Technology, Nanjing 210094, P.R. China}
\email{20201215013@nuist.edu.cn}

\author{Feida Jiang*}
\address{School of Mathematics and Shing-Tung Yau Center of Southeast University, Southeast University, Nanjing 211189, P.R. China}
\email{jfd2001@163.com}

\author{BaoHua Dong}
\address{College of Mathematics and Statistics, Nanjing University of Information Science and Technology, Nanjing 210094, P.R. China}
\email{baohuadong@nuist.edu.cn}

\thanks{This work was supported by the National Natural Science Foundation of China (No. 11771214, No. 11901303).}


\date{\today}
\thanks{*corresponding author}


\begin{abstract}
In this paper, the existence and multiplicity of nontrivial radial convex solutions to general coupled system of $k_i$-Hessian equations in a unit ball are studied via a fixed-point theorem. In particular, we obtain the uniqueness of nontrivial radial convex solution and nonexistence of nontrivial radial $\boldsymbol{k}$-admissible solution to a power-type system coupled by $k_i$-Hessian equations in a unit ball. Moreover, using a generalized Krein-Rutman theorem, the existence of  $\boldsymbol{k}$-admissible solutions to an eigenvalue problem in a general strictly $(k-1)$-convex domain is also obtained.
	
\end{abstract}

\maketitle
\section{Introduction}\label{Section 1}

In this paper, we consider the existence and multiplicity of nontrivial radial $\boldsymbol{k}$-admissible solutions to the coupled system of the following $k_i$-Hessian equations:
\begin{equation}\label{EQUATION1}
\left\{
	\begin{aligned}
		&S_{k_{1}}\left( D^{2}u_{1}\right) =f_{1}\left( |x|,-u_{2}\right) ,& {\rm in}\ B,\\
		&S_{k_{2}}\left( D^{2}u_{2}\right) =f_{2}\left( |x|,-u_{3}\right) ,& {\rm in}\ B,\\
		&\qquad \qquad \quad \vdots\\
		&S_{k_{n-1}}\left( D^{2}u_{n-1}\right) =f_{n-1}\left( |x|,-u_{n}\right) ,& {\rm in}\ B,\\
		&S_{k_{n}}\left( D^{2}u_{n}\right) =f_{n}\left( |x|,-u_{1}\right) ,& {\rm in}\ B,\\
		&u_{i}=0,i=1,\ldots,n,& {\rm on}\ \partial B,
	\end{aligned}
	\right.
\end{equation}
where $\boldsymbol{k}=(k_1,\ldots,k_n)$, $k_{i}=1,\ldots,N\ (i=1,\ldots,n)$,  $B=\left\lbrace x \in\mathbb{R}^{N}:|x|<1\right\rbrace$ is a unit ball, $n\ge 2$ and $N\ge2$ are integers. The nonlinearities $f_{i}\ (i=1,\ldots,n)$ satisfy 
\begin{equation*}
	({\rm F}): f_{i}\in C\left(  [0,1] \times[ 0,+\infty) ,[ 0,+\infty )\right) , \quad i=1,\ldots,n
\end{equation*}
and each $f_i$ is not identical to zero.

The $k$-Hessian operator $S_{k}$ is defined by the $k$-th elementary symmetric function of eigenvalues of $D^2u$, i.e.
\begin{equation*}
	S_{k}\left(D^{2}u\right) := S_{k}\left(\lambda(D^{2}u)\right) =\sum_{1\le i_{1}<\cdots< i_{k}\le N}\lambda_{i_{1}}\cdots\lambda_{i_{k}} , \quad k=1,\ldots,N,
\end{equation*}
where $\lambda(D^{2}u)=(\lambda_1,\ldots,\lambda_N)$ is the vector of eigenvalues of  $ D^{2}u=\left[ \frac{\partial^{2}u}{\partial x_{i}\partial x_{j}}\right] _{n\times n}$, (see \cite{TW-I,Wang-I} for instance).
Notice that when $k=1$, the Hessian operator reduces to the classical Laplace operator $S_{1}(D^{2}u)=\sum_{i=1}^{N}\lambda_{i}=\Delta u$.
When $k=N$, the Hessian operator is the Monge-Amp\`{e}re operator  $S_{N}(D^{2}u)=\prod_{i=1}^{N}\lambda_{i}=\det(D^{2}u)$.
In fact, the $k$-Hessian operator can be regarded as an extension of the Laplace operator and the Monge-Amp\`{e}re operator.
When $k\ge 2$, the $k$-Hessian operator is a fully nonlinear operator. 

Let $u\in C^{2}(\Omega)$ and $\sigma_{k}=\left\lbrace \lambda\in\mathbb{R}^{N}:S_{l}(\lambda)>0,\forall l=1,\ldots,k\right\rbrace $ be a convex cone and its vertex be the origin. 
If $\lambda(D^{2}u)\in\overline{\sigma}_{k}\ (\sigma_{k})$, $u$ is said to be $k$-convex (uniformly $k$-convex) in $\Omega$. 
Equivalently,  if $\lambda(-D^{2}u) \in \overline{\sigma}_{k}\ (\sigma_{k})$, $u$ is $k$-concave (uniformly $k$-concave) in $\Omega$. 
We say $u\in C^{2}(\Omega)\cup C^{0}(\overline{\Omega})$ is $k$-admissible if $\lambda(D^{2}u)\in\overline{\sigma}_{k}$. In particular, an $N$-admissible function $u$ satisfying $\lambda(D^{2}u)\in\overline{\sigma}_{N}$ is said to be convex. 
It is clear that $\sigma_{N}\subset\cdots\subset\sigma_{k}\subset\cdots\subset\sigma_{1}$, which implies that convex functions are contained in $k$-admissible functions.
Actually, we know from \cite{CNS} that for a $k$-Hessian equation, it is elliptic when restricted to $k$-admissible functions.
For $\boldsymbol{k}=(k_{1},\ldots,k_{n})$, $\boldsymbol{u}=(u_{1},\ldots,u_{n})$, if $u_{i}$ is $k_{i}$-admissible and satisfies \eqref{EQUATION1} for all $i=1,\ldots ,n$, we say $\boldsymbol{u}$ is a $\boldsymbol{k}$-admissible solution of \eqref{EQUATION1}. 

Recalling that $f_{i}\in [0,+\infty)\ (i=1,\ldots,n)$ and $\boldsymbol{u}\in C^{2}(B)$ is $\boldsymbol{k}$-admissible solution of \eqref{EQUATION1} vanishing on the boundary, we can achieve that $\boldsymbol{u}$ is sub-harmonic in $B$ from \cite{Wang-I}.
Hence, we apply the maximum principle to conclude that $\boldsymbol{u}$ is negative in $B$.

The study of $k$-Hessian equations plays an important role in differential geometry, fluid mechanics and other applied disciplines. 
In the past years, many authors show great interest in solutions of $k$-Hessian equations and many excellent results on $k$-Hessian equations have been obtained, for instance, see \cite{BLL,Guan,Jian,JTX,JT,MQ,Tru,Wang-I,CW,CNS,TW-I,TW-II,T}. 
However, there are few studies that consider the fully nonlinear coupled systems except
\cite{FZ-I,WHY,ZZ,WYZB,CJX,FMQ-I,FMQ} based on our cognition.
For example, by using fixed point theorem, Wang \cite{WHY} established the existence, multiplicity and nonexistence of convex radial solutions to a coupled system of Monge-Amp\`{e}re equations in superlinear and sublinear cases.
In \cite{FZ-I}, the authors studied the existence and multiplicity of nontrivial radial solutions for system coupled by multiparameter $k$-Hessian equations and obtained sufficient conditions for the existence of nontrivial radial solutions to power-type coupled $k$-Hessian system based on a eigenvalue theory in cones.  
In particular, Cui considered a Hessian type system coupled by different $k$-Hessian equations and obtained the existence of entire $k$-convex radial solutions, see \cite{CJX}. 

Inspired by the above works, we are interested in a system coupled by  different $k$-Hessian equations with general nonlinearities which satisfy $\alpha_{i}$ or $\beta_{i}$-asymptotic growth conditions. 
In this paper, we shall establish the existence and multiplicity of nontrivial radial $\boldsymbol{k}$-admissible solutions of the weakly coupled degenerated system \eqref{EQUATION1}. 
It is worth to notice that the system \eqref{EQUATION1} contains a variety of different $k$-Hessian equations which is significantly different from that in \cite{WYZB,FZ-I,ZZ,FMQ-I} such that the problem we considered
can contain Laplace equations and Monge-Amp\`{e}re equations at the same time. 
This kind of system can represent the coupling of different types of elliptic equations, which makes our problem more comprehensive and more applicable.

If $\alpha_{i}, \beta_{i}>0$, we let 
$$\underline{f}_{i}^{0}=\liminf_{c \to 0^{+}} \min_{0\le t\le 1}\frac{f_{i}(t,c)}{c^{\alpha _{i}}}, \qquad \qquad\underline{f}_{i}^{\infty}=\liminf_{c \to \infty} \min_{0\le t\le 1}\frac{f_{i}(t,c)}{c^{\beta _{i}}},$$

$$\overline{f}_{i}^{0}=\limsup_{c \to 0^{+}} \max_{0\le t\le 1}\frac{f_{i}(t,c)}{c^{\alpha _{i}}}, \qquad \qquad\overline{f}_{i}^{\infty}=\limsup_{c \to \infty} \max_{0\le t\le 1}\frac{f_{i}(t,c)}{c^{\beta_{i}}}.$$
Here, we call them $\alpha_{i}$ or $\beta_{i}$-asymptotic growth condition, super-$\alpha_{i}$ or $\beta_{i}$-asymptotic growth condition and sub-$\alpha_{i}$ or $\beta_{i}$-asymptotic growth condition.
Compared with some $N$-asymptotic growth  (see, for instance \cite{GS,WHY,FMQ} where the constants $\alpha_{i}=\beta_{i}=N$) in studying Monge-Amp\`{e}re equations and some $k$-asymptotic growth (see, for instance \cite{FZ-I,ZXM-I,ZXM-II} where the constants $\alpha_{i}=\beta_{i}=k$) in studying $k$-Hessian equations, our conditions are more flexible.
By imposing suitable conditions on $\underline{f}_{i}^{0}$, $\underline{f}_{i}^{\infty}$, $\overline{f}_{i}^{0}$, $\overline{f}_{i}^{\infty}$ and coordinating inequality relations between $\alpha_{i},\beta_{i}$ and $k_{i}$, we obtain existence and multiplicity results in general cases as follows.

We will assume $\boldsymbol{f}=\{f_1,\ldots,f_n\}$ satisfies one of the following conditions:\\
(C1) $\underline{f}_{i}^{0}, \overline{f}_{i}^{\infty}\in (0,+\infty)$, $i=1,\ldots,n$ and $f_{i}(t,0)=0$, $i=2,\ldots,n$;\\
(C2) $\overline{f}_{i}^{0}, \underline{f}_{i}^{\infty}\in (0,+\infty)$, $i=1,\ldots,n$;\\
(C3) $\underline{f}_{i}^{0}, \underline{f}_{i}^{\infty}\in (0,+\infty)$, $i=1,\ldots,n$ and $f_{i}(t,0)=0$, $i=2,\ldots,n$;\\
(C4) $\overline{f}_{i}^{0}, \overline{f}_{i}^{\infty}\in (0,+\infty)$, $i=1,\ldots,n$.
\begin{Theorem}\label{Th1.1}(Existence theorem)
Suppose that {\rm (F)}  and one of the following conditions hold:\\
(a). {\rm (C1)} holds and positive constants $\alpha_{i}$, $\beta_{i}\ (i=1,\ldots,n)$ satisfy $$\prod_{i=1}^{n}\alpha_{i}<\prod_{i=1}^{n}k_{i},\quad\prod_{i=1}^{n}\beta_{i}<\prod_{i=1}^{n}k_{i};$$\\
(b). {\rm (C2)} holds and positive constants $\alpha_{i}$, $\beta_{i}\ (i=1,\ldots,n)$ satisfy $$\prod_{i=1}^{n}\alpha_{i}>\prod_{i=1}^{n}k_{i},\quad\prod_{i=1}^{n}\beta_{i}>\prod_{i=1}^{n}k_{i}.$$
Then system \eqref{EQUATION1} has at least one nontrivial radial convex solution.
\end{Theorem}

Theorem \ref{Th1.1} is concerning the existence of nontrivial radial convex solutions to the weakly coupled degenerate system  \eqref{EQUATION1} with general nonlinear terms. Furthermore, we can consider the result of multiplicity as well.

Let
\begin{equation*}
	\begin{aligned}
		&G_{i}=\max\left\lbrace f_{i}(t,v_{i+1}(t)):(t,v_{i+1}(t))\in [0,1]\times[0,G_{i+1}^\frac{1}{k_{i+1}}] \right\rbrace, \quad i=1,\ldots,n-1,\\
		&G_{n}=\max\left\lbrace f_{n}(t,v_{1}(t)):(t,v_{1}(t))\in [0,1]\times[0,\frac{r_{0}}{4}] \right\rbrace,\\
		&\tilde{G}_{i}=\max\left\lbrace  f_{i}(t,v_{i+1}(t)):
		(t,v_{i+1}(t))\in[0,1]\times[0,\tilde{G}_{i+1}^\frac{1}{k_{i+1}}]
		\right\rbrace, \quad i=2,\ldots,n-1, \\
		&\tilde{G}_{n}=\max\left\lbrace  f_{n}(t,v_{1}(t)):
		(t,v_{1}(t))\in[0,1]\times[0,R_{0}]
		\right\rbrace,\\
		&E_{i}=\min\left\lbrace 
		f_{i}(t,v_{i+1}(t)):(t,v_{i+1}(t))\in [\frac{1}{4},\frac{3}{4}]\times[\frac{1}{4}\Gamma_{i+1}E_{i+1}^\frac{1}{k_{i+1}},\tilde{G}_{i+1}^\frac{1}{k_{i+1}}] 
		\right\rbrace, \quad i=1,\ldots,n-1,\\
		&E_{n}=\min\left\lbrace 
		f_{n}(t,v_{1}(t)):(t,v_{1}(t))\in [\frac{1}{4},\frac{3}{4}]\times[\frac{1}{4}R_{0},R_{0}] 
		\right\rbrace.
	\end{aligned}
\end{equation*}

\begin{Theorem}\label{Th1.2}(Multiplicity theorem)
    Suppose that {\rm (F)}  and one of the following conditions hold: \\
	(c). {\rm (C3)} holds, positive constants $\alpha_{i}$, $\beta_{i}\ (i=1,\ldots,n)$ satisfy $$\prod_{i=1}^{n}\alpha_{i}<\prod_{i=1}^{n}k_{i},\quad\prod_{i=1}^{n}\beta_{i}>\prod_{i=1}^{n}k_{i},$$
	and there exists a positive constant $r_{0}$ such that $r_{0}>G_{1}^\frac{1}{k_{1}}$;\\
    (d). {\rm (C4)} holds, positive constants $\alpha_{i}$, $\beta_{i}(i=1,\ldots,n)$ satisfy $$\prod_{i=1}^{n}\alpha_{i}>\prod_{i=1}^{n}k_{i},\quad\prod_{i=1}^{n}\beta_{i}<\prod_{i=1}^{n}k_{i},$$
    and there exists a positive constant $R_{0}$ such that $R_{0}<\Gamma_{1}E_{1}^\frac{1}{k_{1}}.$
Then system \eqref{EQUATION1} has at least two nontrivial radial convex solutions.
\end{Theorem} 

\begin{remark}
	Theorem \ref{Th1.1} and  Theorem \ref{Th1.2} show that the existence and multiplicity of nontrivial radial convex solutions to system \eqref{EQUATION1} respectively, see the penultimate paragraph in this section for more detailed explanations of convex solutions. Since the convex solutions are contained in the $\boldsymbol{k}$-admissible solutions, Theorems \ref{Th1.1} and \ref{Th1.2} show the existence and multiplicity of $\boldsymbol{k}$-admissible solutions as well.
\end{remark}
\begin{remark}  
	It is worth to mention that the condition (C1) and (C3) can be replaced by
	$f_{m_k}(t,0)=0,\ k=1,\ldots,n-1$, where $\{m_1,\ldots,m_{n-1}\}\subset\{1,\ldots,n\}$ and we describe as $f_{i}(t,0)=0$, $i=2,\ldots,n$ for the sake of proof.
\end{remark}

Specifically, we also study the uniqueness and nonexistence of nontrivial radial solutions to a power-type coupled system of $\boldsymbol{k}$-Hessian equations:
\begin{equation}\label{EQUATION2}
	\left\{
	\begin{aligned}
		&S_{k_{1}}\left( D^{2}u_{1}\right) =\left(-u_{2}\right)^{\gamma_1} ,& {\rm in}\ B,\\
		&S_{k_{2}}\left( D^{2}u_{2}\right) =\left(-u_{3}\right)^{\gamma_2} ,& {\rm in}\ B,\\
		&\qquad \qquad \quad \vdots\\
		&S_{k_{n-1}}\left( D^{2}u_{n-1}\right) =\left(-u_{n}\right)^{\gamma_{n-1}} ,& {\rm in}\ B,\\
		&S_{k_{n}}\left( D^{2}u_{n}\right) =\left(-u_{1}\right)^{\gamma_n} ,& {\rm in}\ B,\\
		&u_{i}=0,i=1,\ldots,n,& {\rm on}\ \partial B,
	\end{aligned}
	\right.
\end{equation}
where $\gamma_i\ (i=1,\ldots,n)$ are positive constants.

It is obvious that system \eqref{EQUATION2} is a special case of system \eqref{EQUATION1}. By the definitions of $\alpha_{i}$ or $\beta_{i}$-asymptotic growth condition, the growth of nonlinearities of the power-type system \eqref{EQUATION2} satisfies $\alpha_{i}=\beta_{i}=\gamma_i$, which asserts the existence of nontrivial radial convex solutions to system \eqref{EQUATION2} by Theorem \ref{Th1.1} when $\prod_{i=1}^{n}\gamma_i\ne \prod_{i=1}^{n}k_i$.
Next, we go further to study the uniqueness of nontrivial radial convex solution to system \eqref{EQUATION2} in Theorem \ref{Th1.3}.

\begin{Theorem}\label{Th1.3}(Uniqueness theorem)
Suppose that positive constant $\prod_{i=1}^{n}\gamma_i$ satisfies
$$\prod_{i=1}^{n}\gamma_i<\prod_{i=1}^{n}k_i,$$
then system \eqref{EQUATION2} has a unique nontrivial radial convex solution.
\end{Theorem}

Here, we get the uniqueness result of nontrivial radial convex solution to system \eqref{EQUATION2} in the assumption of $\prod_{i=1}^{n}\gamma_i<\prod_{i=1}^{n}k_i$.  
Besides, we obtain the nonexistence of nontrivial radial $\boldsymbol{k}$-admissible solution in $B$ when $\prod_{i=1}^{n}\gamma_i=\prod_{i=1}^{n}k_i$.

\begin{Theorem}\label{Th1.4}(Nonexistence theorem)
	Suppose that positive constant $\prod_{i=1}^{n}\gamma_i$ satisfies
	$$\prod_{i=1}^{n}\gamma_i=\prod_{i=1}^{n}k_i,$$
	then system \eqref{EQUATION2} admits no nontrivial radial $\boldsymbol{k}$-admissible solution.
\end{Theorem}

When $\prod_{i=1}^{n}\gamma_i=\prod_{i=1}^{n}k_i$, we are interested in the existence of nonzero $\boldsymbol{k}$-admissible solutions for the eigenvalue problem:
\begin{equation}\label{EQUATION3}
	\left\{
	\begin{aligned}
		&S_{k_{1}}\left( D^{2}u_{1}\right) =\lambda_1 \left(-u_{2}\right)^{\gamma_1} ,& {\rm in}\ \Omega,\\
		&S_{k_{2}}\left( D^{2}u_{2}\right) =\lambda_2 \left(-u_{3}\right)^{\gamma_2} ,& {\rm in}\ \Omega,\\
		&\qquad \qquad \quad \vdots\\
		&S_{k_{n-1}}\left( D^{2}u_{n-1}\right) =\lambda_{n-1} \left(-u_{n}\right)^{\gamma_{n-1}} ,& {\rm in}\ \Omega,\\
		&S_{k_{n}}\left( D^{2}u_{n}\right) =\lambda_n \left(-u_{1}\right)^{\gamma_n} ,& {\rm in}\ \Omega,\\
		&u_{i}=0,i=1,\ldots,n,& {\rm on}\ \partial \Omega,
	\end{aligned}
	\right.
\end{equation}
with positive parameters $\lambda_{i}\ (i=1,\ldots,n)$, where $\Omega \in\mathbb{R}^{N}$ is a bounded, smooth and strictly $(k-1)$-convex domain, $N\ge2$.

In fact, Wang has proved the existence of a positive eigenvalue $\lambda^*$ for a single $k$-Hessian equation with $f(u)=\lambda|u|^{k}(k<N)$ in \cite{Wang-I}. When $\lambda=\lambda^*$, the corresponding eigenfunction $\varphi^*$ is nonzero $k$-admissible and that any other eigenfunction would be a positive constant multiple of $\varphi^*$. 
Since $\lambda^*$ acts like a bifurcation point for system \eqref{EQUATION3}, we can be reminiscent of the generalized Krein-Rutman theorem in \cite{Jacobsen} to obtain the existence of $\boldsymbol{k}$-admissible solutions to eigenvalue problem \eqref{EQUATION3}.

\begin{Theorem}\label{Th1.5}(Eigenvalue problem)
Suppose that $\Omega \in\mathbb{R}^{N}$ is a bounded, smooth and strictly $(k-1)$-convex domain, positive constant $\prod_{i=1}^{n}\gamma_i$ satisfies $$\prod_{i=1}^{n}\gamma_i=\prod_{i=1}^{n}k_i,$$
then system \eqref{EQUATION3} admits a nonzero $\boldsymbol{k}$-admissible solution if and only if $\lambda_1\lambda_2^{\frac{\gamma_1}{k_2}}\cdots\lambda_n^{\frac{\prod_{i=1}^{n-1}\gamma_i}{\prod_{i=2}^{n}k_i}}=\lambda_0^{k_1}$, where $\lambda_0 \neq 1$  is a positive constant, such that the system
\begin{equation}\label{1.4}
	\left\{
	\begin{aligned}
		&S_{k_{1}}\left( D^{2}(\frac{u_{1}}{\lambda_0})\right) =\left(-u_{2}\right)^{\gamma_1} ,& {\rm in}\ \Omega,\\
		&S_{k_{2}}\left( D^{2}u_{2}\right) =\left(-u_{3}\right)^{\gamma_2} ,& {\rm in}\ \Omega,\\
		&\qquad \qquad \quad \vdots\\
		&S_{k_{n-1}}\left( D^{2}u_{n-1}\right) =\left(-u_{n}\right)^{\gamma_{n-1}} ,& {\rm in}\ \Omega,\\
		&S_{k_{n}}\left( D^{2}u_{n}\right) =\left(-u_{1}\right)^{\gamma_n} ,& {\rm in}\ \Omega,\\
		&u_{i}=0, \quad i=1,\ldots,n,& {\rm on}\ \partial \Omega,
	\end{aligned}
	\right.
\end{equation}
has a nonzero $\boldsymbol{k}$-admissible solution.
\end{Theorem}

Note that the existence of nonzero $\boldsymbol{k}$-admissible solution of \eqref{1.4} is guaranteed by a generalized Krein-Rutman theorem, see Section \ref{Section 5} for details.

In this article, we study the existence and multiplicity of radial convex solutions to system \eqref{Th1.1}, the uniqueness of radial convex solution and nonexistence of radial $\boldsymbol{k}$-admissible solution to system \eqref{Th1.2}, and the existence of radial $\boldsymbol{k}$-admissible solutions to the related eigenvalue problem \eqref{EQUATION3}.
The reasons why Theorems \ref{Th1.1}, \ref{Th1.2} and \ref{Th1.3} are only restricted to the convex solutions will be further explained in Remarks \ref{Rm2.1} and \ref{Rm4.1}.
The improvement from convex solutions to $\boldsymbol{k}$-admissible solutions in Theorem \ref{Th1.1}, Theorem \ref{Th1.2} and Theorem \ref{Th1.3} is still an interesting problem, which attracts us to find another way or technique to solve this problem in a sequel.

The rest of the paper is organized as follows. 
In Section \ref{Section 2}, we make some preliminary calculations of $C^2$ radial solutions and present a fixed point theorem in Theorem \ref{Lemma1.1}. 
In Section \ref{Section 3}, we give the proof of existence and multiplicity results for system \eqref{EQUATION1} with general nonlinearities by using the fixed point theorem.
In Section \ref{Section 4}, the uniqueness and nonexistence results for power-type coupled system \eqref{EQUATION2} which is a special case of \eqref{EQUATION1} are considered. 
In Section \ref{Section 5}, by overcoming the difficulties caused by verifying the condition of generalized Krein-Rutman theorem which to prove the operator is strong, we obtain the existence of nonzero $\boldsymbol{k}$-admissible solutions to the eigenvalue problem \eqref{EQUATION3} in a general strictly $(k-1)$-convex domain.

\vspace{3mm}

\section{Preliminaries}\label{Section 2}

To study radial classical solutions of system \eqref{EQUATION1}, we assume $u(|x|)=u(t)$ be the radial function with $t=\sqrt{\sum_{i=1}^{N}x_{i}^{2}}$, then it follows from Lemma 2.1 in \cite{JB} that the $k$-Hessian operator becomes 
\begin{equation*}
S_{k}(D^{2}u)=C_{N-1}^{k-1} u^{\prime\prime}(t)(\frac{u^{\prime}(t)}{t})^{k-1}+C_{N-1}^{k}(\frac{u^{\prime}(t)}{t})^{k},\quad t\in (0,1).
\end{equation*}

Then we can convert \eqref{EQUATION1} to the following system of ordinary differential equations:
\begin{equation}\label{EQUATION2.1}
 \left\{
	\begin{array}{ll}
	C_{N-1}^{k_{1}-1}u_{1}^{\prime\prime}(t)(\frac{u_{1}^{\prime}(t)}{t})^{k_{1}-1}+C_{N-1}^{k_{1}}(\frac{u_{1}^{\prime}(t)}{t})^{k_{1}}=f_{1}(t,-u_{2}),\quad&0<t<1,\\[1ex]
	
	C_{N-1}^{k_{2}-1}u_{2}^{\prime\prime}(t)(\frac{u_{2}^{\prime}(t)}{t})^{k_{2}-1}+C_{N-1}^{k_{2}}(\frac{u_{2}^{\prime}(t)}{t})^{k_{2}}=f_{2}(t,-u_{3}),\quad&0<t<1,\\
	
	\qquad \qquad\vdots\\
	
	C_{N-1}^{k_{n-1}-1}u_{n-1}^{\prime\prime}(t)(\frac{u_{n-1}^{\prime}(t)}{t})^{k_{n-1}-1}+C_{N-1}^{k_{n-1}}(\frac{u_{n-1}^{\prime}(t)}{t})^{k_{n-1}}=f_{n-1}(t,-u_{n}),&0<t<1,\\[1ex]
	
	C_{N-1}^{k_{n}-1}u_{n}^{\prime\prime}(t)(\frac{u_{n}^{\prime}(t)}{t})^{k_{n}-1}+C_{N-1}^{k_{n}}(\frac{u_{n}^{\prime}(t)}{t})^{k_{n}}=f_{n}(t,-u_{1}),\quad&0<t<1,\\[1ex]
	
	u_{i}(1)=u_{i}^{\prime}(0)=0,\quad i=1,\ldots,n.
	\end{array}
 \right.
\end{equation}

Equivalently, we seek nonnegative $k$-concave solutions for convenience by making a simple transformation $v_{i}=-u_{i}\ (i=1,\ldots,n)$ in \eqref{EQUATION2.1}, which leads to the following system:
\begin{equation}\label{EQUATION2.2}
\left\{
   \begin{array}{ll}
   	C_{N-1}^{k_{1}-1}(-v_{1})^{\prime\prime}(t)(\frac{(-v_{1})^{\prime}(t)}{t})^{k_{1}-1}+C_{N-1}^{k_{1}}(\frac{(-v_{1})^{\prime}(t)}{t})^{k_{1}}=f_{1}(t,v_{2}),&0<t<1,\\[1ex]
   	
   	C_{N-1}^{k_{2}-1}(-v_{2})^{\prime\prime}(t)(\frac{(-v_{2})^{\prime}(t)}{t})^{k_{2}-1}+C_{N-1}^{k_{2}}(\frac{(-v_{2})^{\prime}(t)}{t})^{k_{2}}=f_{2}(t,v_{3}),&0<t<1,\\
   	
   	\qquad \qquad\vdots\\
   	
   	C_{N-1}^{k_{n-1}-1}(-v_{n-1})^{\prime\prime}(t)(\frac{(-v_{n-1})^{\prime}(t)}{t})^{k_{n-1}-1}+C_{N-1}^{k_{n-1}}(\frac{(-v_{n-1})^{\prime}(t)}{t})^{k_{n-1}}=f_{n-1}(t,v_{n}),&0<t<1,\\[1ex]
   	
   	C_{N-1}^{k_{n}-1}(-v_{n})^{\prime\prime}(t)(\frac{(-v_{n})^{\prime}(t)}{t})^{k_{n}-1}+C_{N-1}^{k_{n}}(\frac{(-v_{n})^{\prime}(t)}{t})^{k_{n}}=f_{n}(t,v_{1}),&0<t<1,\\[1ex]
   	
   	v_{i}(1)=v_{i}^{\prime}(0)=0,\quad i=1,\ldots,n.
   \end{array}
\right.	
\end{equation} 

By integration, we get from \eqref{EQUATION2.2} that
\begin{equation*}
\left\{
   \begin{array}{ll}
	v_{1}(t)=\int_t^1\left( \frac{k_{1}}{\tau^{N-k_{1}}}\int_0^\tau \frac{s^{N-1}}{C_{N-1}^{k_{1}-1}}f_{1}\left( s,v_{2}(s)\right)\,{\rm d}s \right)^{\frac{1}{k_{1}} }\,{\rm d}\tau,  &0\le t\le 1, \\
	
    v_{2}(t)=\int_t^1\left( \frac{k_{2}}{\tau^{N-k_{2}}}\int_0^\tau \frac{s^{N-1}}{C_{N-1}^{k_{2}-1}}f_{2}\left( s,v_{3}(s)\right)\,{\rm d}s \right)^{\frac{1}{k_{2}} }\,{\rm d}\tau,  &0\le t\le 1, \\
    
    \qquad\qquad\vdots\\
    
    v_{n-1}(t)=\int_t^1\left( \frac{k_{n-1}}{\tau^{N-k_{n-1}}}\int_0^\tau \frac{s^{N-1}}{C_{N-1}^{k_{n-1}-1}}f_{n-1}\left( s,v_{n}(s)\right)\,{\rm d}s \right)^{\frac{1}{k_{n-1}} }\,{\rm d}\tau,  &0\le t\le 1, \\
    
    v_{n}(t)=\int_t^1\left( \frac{k_{n}}{\tau^{N-k_{n}}}\int_0^\tau \frac{s^{N-1}}{C_{N-1}^{k_{n}-1}}f_{n}\left( s,v_{1}(s)\right)\,{\rm d}s \right)^{\frac{1}{k_{n}} }\,{\rm d}\tau,  &0\le t\le 1.
   \end{array}
\right.	
\end{equation*} 

Considering the Banach space $X:=C[0,1]$, for $\boldsymbol{v}=(v_{1},\ldots,v_{n})\in \underbrace{X \times \cdots \times X}_{n}$, we define $||\boldsymbol{v}||=\sum_{i=1}^{n}||\boldsymbol{v_{i}}(t)||=\sum_{i=1}^{n}\sup\limits_{t\in[0,1]}|v_{i}(t)|$. Let $K$ be a cone in $X$ defined as 
\begin{equation}\label{cone}
	K:=\left\lbrace v\in X:v(t)\ge 0, t\in [0,1],\min_{\frac{1}{4}\le t\le \frac{3}{4}}v(t)\ge \frac{1}{4}||v||\right\rbrace .
\end{equation}

We define the operators $T_{i}:K \to X\ (i=1,\ldots,n)$ to be
\begin{equation*}
	\begin{aligned}
		&T_{1}(v_{2})(t)=\int_t^1\left( \frac{k_{1}}{\tau^{N-k_{1}}}\int_0^\tau \frac{s^{N-1}}{C_{N-1}^{k_{1}-1}}f_{1}\left( s,v_{2}(s)\right)\,{\rm d}s \right)^{\frac{1}{k_{1}} }\,{\rm d}\tau,  \\
		&T_{2}(v_{3})(t)=\int_t^1\left( \frac{k_{2}}{\tau^{N-k_{2}}}\int_0^\tau \frac{s^{N-1}}{C_{N-1}^{k_{2}-1}}f_{2}\left( s,v_{3}(s)\right)\,{\rm d}s \right)^{\frac{1}{k_{2}} }\,{\rm d}\tau,   \\
		&\qquad\qquad\vdots\\
		&T_{n-1}(v_{n})(t)=\int_t^1\left( \frac{k_{n-1}}{\tau^{N-k_{n-1}}}\int_0^\tau \frac{s^{N-1}}{C_{N-1}^{k_{n-1}-1}}f_{n-1}\left( s,v_{n}(s)\right)\,{\rm d}s \right)^{\frac{1}{k_{n-1}} }\,{\rm d}\tau,   \\
		&T_{n}(v_{1})(t)=\int_t^1\left( \frac{k_{n}}{\tau^{N-k_{n}}}\int_0^\tau \frac{s^{N-1}}{C_{N-1}^{k_{n}-1}}f_{n}\left( s,v_{1}(s)\right)\,{\rm d}s \right)^{\frac{1}{k_{n}} }\,{\rm d}\tau.
	\end{aligned}
\end{equation*} 

Note that each image of operator is a nonnegative $k$-concave function on $[0,1]$ and we define $T_{1}(v_{2})=v_{1},T_{2}(v_{3})=v_{2},\cdots,T_{n}(v_{1})=v_{n}$ in $K$. Thus, by the concavity of $v_i\ (i=1, \ldots, n)$, it is easy to see that $T_{i}\ (i=1,\ldots,n)$ maps $K$ into itself. Besides, by standard arguments, we know that every operator is completely continuous.

Next, we define a composite operator $Tv_{1}=T_{1}T_{2} \cdots T_{n}(v_{1})$, which is also completely continuous from $K$ to $K$. We can see that positive solutions of \eqref{EQUATION2.2} are equivalent to nonzero fixed points of operator $T$ in cone $K$. If $\boldsymbol{v}=(v_{1},\ldots,v_{n})\in \underbrace{C[0,1] \times \cdots \times C[0,1]}_{n}$ is a positive solution of \eqref{EQUATION2.2}, then $v_{1}$ must be a nonzero fixed point of $T$ in $K$; conversely if $v_{1}\in K\setminus\left\lbrace 0  \right\rbrace $ is a fixed point of $T$, we can define $v_{n}=T_{n}(v_{1}),v_{n-1}=T_{n-1}(v_{n}),\cdots,v_{2}=T_{2}(v_{3})$ such that $(v_{1},\ldots,v_{n})\in \underbrace{C[0,1] \times \cdots \times C[0,1]}_{n}$ solves \eqref{EQUATION2.2}.

\begin{remark}\label{Rm2.1}
	As we shall see in the last two paragraphs, we let each $T_i\ (i=1\ldots,n)$ maps $K$ to itself which implies that $v^{\prime}(t)=(-u)^{\prime}(t)$ is nonincreasing from Lemma 2.2 in \cite{WHY}.
	On the other hand, the eigenvalues of the second derivative of radial classical function in a unit ball can be represented by $\lambda(D^2u)=(u^{\prime\prime}(t),\frac{u^{\prime}(t)}{t},\ldots,\frac{u^{\prime}(t)}{t}),t\in[0,1]$, we combine this with the definition of $\boldsymbol{k}$-admissible function, an immediate consequence is that we essentially achieve the $(N-1)$-admissible function in $\mathbb{R}^{N}$, that is, all $\frac{u^{\prime}(t)}{t}\ge 0$. 	
	To sum up, all eigenvalues of the Hessian matrix of nontrivial radial $\boldsymbol{k}$-admissible solutions of system \eqref{EQUATION1} and system \eqref{EQUATION2} are nonnegative and exist in its closure of convex cone, which can draw our conclusion.
\end{remark}

The proofs of our existence and multiplicity results are based on the following well-known fixed point theorem of cone, (see Theorem 2.3.4 in Guo \cite{Guo}).
\begin{Theorem}\label{Lemma1.1}
	Let $X$ be a Banach space and $K$ is a cone in $X$. Assume that $\Omega_{1}$, $\Omega_{2}$ are bounded open subsets of $X$ with $0\in\Omega_{1}$, $\overline{\Omega}_{1}\subset\Omega_{2}$ and let
	$$T: K\cap(\overline{\Omega}_{2} \setminus \Omega_{1})\to K$$
	be completely continuous such that either
\begin{enumerate}[(i)]
	\item  $||Tu||\ge||u||$, $u\in K\cap \partial\Omega_{1}$ and $||Tu||\le||u||$, $u\in K\cap \partial\Omega_{2}$; or
	\item  $||Tu||\le||u||$, $u\in K\cap \partial\Omega_{1}$ and $||Tu||\ge||u||$, $u\in K\cap \partial\Omega_{2}$
	\end{enumerate}
holds, where $||\cdot||$ is a norm in X, $\Omega_{R}=\{u\in K:||u||<R\}$ and $\partial\Omega_{R}=\{u\in K:||u||=R\}$. Then T has a fixed point in $K\cap(\overline{\Omega}_{2} \setminus \Omega_{1})$.
\end{Theorem}

\vspace{3mm}

\section{Existence and multiplicity}\label{Section 3}

In this section, we apply the fixed point theorem of cone in Theorem \ref{Lemma1.1} to prove the existence and multiplicity results in Theorem \ref{Th1.1} and Theorem \ref{Th1.2}. To simplify notation, we denote $v_1$ by $v_{n+1}$ .
\subsection{Existence}\label{Section 3.1}

In order to prove the Theorem \ref{Th1.1}, we first introduce two useful lemmas.

\begin{Lemma}\label{lower bound}
	Assume {\rm (F)} holds. Let $\eta,m>0$ and $v_{i}\in K,\ i=1,\ldots,n$. If for any $t\in [\frac{1}{4},\frac{3}{4}]$ and $i=1,\ldots,n$, we have
	$$f_{i}\left( t,v_{i+1}(t)\right) \ge \eta v_{i+1}^{m}(t),$$
then
\begin{equation*} 
        T_{i}(v_{i+1})(\frac{1}{4})\ge \Gamma_{i}\eta^{\frac{1}{k_{i}}}(\frac{1}{4})^{\frac{m}{k_{i}}}||v_{i+1}||^{\frac{m}{k_{i}}}, \quad i=1,\ldots,n,
\end{equation*}
where $\Gamma_{i}$ are positive constants given by  $\Gamma_{i}=\int_{\frac{1}{4}}^{\frac{3}{4}}\left( \frac{k_{i}}{\tau^{N-k_{i}}}\int_{\frac{1}{4}}^{\tau}\frac{s^{N-1}}{C_{N-1}^{k_{i}-1}}\,{\rm d}s \right)^{\frac{1}{k_{i}}}\,{\rm d}\tau ,i=1,\ldots,n$.
\end{Lemma}
\begin{proof}
	For $v_{1}\in K$, we have
\begin{equation*}
   \begin{split}
	T_{n}(v_{1})(\frac{1}{4})
	=& \int_{\frac{1}{4}}^1\left( \frac{k_{n}}{\tau^{N-k_{n}}}\int_0^\tau \frac{s^{N-1}}{C_{N-1}^{k_{n}-1}}f_{n}\left( s,v_{1}(s)\right)\,{\rm d}s \right)^{\frac{1}{k_{n}} }\,{\rm d}\tau\\
	\ge& \int_{\frac{1}{4}}^{\frac{3}{4}}\left( \frac{k_{n}}{\tau^{N-k_{n}}}\int_{\frac{1}{4}}^\tau \frac{s^{N-1}}{C_{N-1}^{k_{n}-1}} \eta v_{1}^{m}(s)\,{\rm d}s \right)^{\frac{1}{k_{n}} }\,{\rm d}\tau\\
	\ge& \int_{\frac{1}{4}}^{\frac{3}{4}}\left( \frac{k_{n}}{\tau^{N-k_{n}}}\int_{\frac{1}{4}}^\tau \frac{s^{N-1}}{C_{N-1}^{k_{n}-1}} \eta \left( \frac{1}{4}||v_{1}||\right)^{m} \,{\rm d}s \right)^{\frac{1}{k_{n}} }\,{\rm d}\tau\\
	=&\Gamma_{n}\eta^{\frac{1}{k_{n}}}\left( \frac{1}{4}||v_{1}||\right)^{\frac{m}{k_{n}}}.
  \end{split}	
\end{equation*}
For $v_{i}\in K\ (i=2,\ldots,n)$, we have similar calculations. Here we omit them for simplicity.
\end{proof}

\begin{Lemma}\label{upper bound}
	Assume {\rm (F)} holds. Let $\varepsilon,d>0$ and $v_{i}\in K,i=1,\ldots,n$. If for any $t\in [0,1]$ and $i=1,\ldots,n$, we have
    $$f_{i}\left( t,v_{i+1}(t)\right) \le \varepsilon v_{i+1}^{d}(t),$$
	then
\begin{equation*} 
    T_{i}(v_{i+1})(t)<\left( \varepsilon||v_{i+1}||^{d}\right) ^{\frac{1}{k_{i}}}, \quad i=1,\ldots,n.
\end{equation*}
\end{Lemma}
\begin{proof}
		Since $v_{1}(t)\in K$, ${\forall t}$ $\in[0,1]$, we have
	\begin{equation}\label{Lemma3.2}
		\begin{split}
			T_{n}(v_{1})(t)
			\le& \int_0^1\left( \frac{k_{n}}{\tau^{N-k_{n}}}\int_0^\tau \frac{s^{N-1}}{C_{N-1}^{k_{n}-1}}f_{n}\left( s,v_{1}(s)\right)\,{\rm d}s \right)^{\frac{1}{k_{n}} }\,{\rm d}\tau\\
			\le& \int_0^1\left( \frac{k_{n}}{\tau^{N-k_{n}}}\int_0^\tau \frac{s^{N-1}}{C_{N-1}^{k_{n}-1}}\varepsilon v_{1}^{d}(s)\,{\rm d}s\right)^{\frac{1}{k_{n}} }\,{\rm d}\tau\\
			\le& \frac{1}{2}\left( \frac{k_{n}}{NC_{N-1}^{k_{n}-1}}\right)^{\frac{1}{k_{n}}}\left( \varepsilon||v_{1}||^{d}\right)^{\frac{1}{k_{n}}} \\
			<&\left(\varepsilon||v_{1}||^{d} \right) ^{\frac{1}{k_{n}}},
		\end{split}	
	\end{equation}
where the fact $\frac{1}{2}\left( \frac{k_{n}}{NC_{N-1}^{k_{n}-1}}\right)^{\frac{1}{k_{n}}}<1$ is used in the last inequality, which is easily checked.
For $v_{i}\in K\ (i=2,\ldots,n)$, we also have similar conclusions.
\end{proof}

On the basis of the above preparations, we give the proof for the existence result in Theorem \ref{Th1.1} with the aid of the fixed point theorem of cone.
\begin{proof}[Proof of Theorem \ref{Th1.1}]
(a). It follows from $\underline{f}_{i}^{0}\in (0,+\infty)\ (i=1,\ldots,n)$ that for any given $\varepsilon_{1}\in  (0,\min\{\underline{f}_{i}^{0},i=1,\ldots,n\})$, there exists a constant $r_{1}\in(0,1)$ such that
\begin{equation}\label{Equation3.3}
f_{i}\left( t,v_{i+1}(t)\right)  \ge (\underline{f}_{i}^{0}-\varepsilon_{1})v_{i+1}^{\alpha_{i}}, \quad 0\le v_{i+1}\le r_{1},
\end{equation}
for any $t\in[0,1]$ and $i=1,\ldots,n$.
Let 
\begin{equation*}
	L_{1}:=\Gamma_{1}\cdots \Gamma_{n}^{\frac{\prod_{i=1}^{n-1}\alpha_{i}}{\prod_{i=1}^{n-1}k_{i}}}
	( \underline{f}_{1}^{0}-\varepsilon_{1} )^{\frac{1}{k_{1}}} 
	\cdots
	( \underline{f}_{n}^{0}-\varepsilon_{1} )^{\frac{\prod_{i=1}^{n-1}\alpha_{i}}{\prod_{i=1}^{n}k_{i}}}
	( \frac{1}{4}) ^{\frac{\alpha_{1}}{k_{1}}+\cdots+\frac{\prod_{i=1}^{n}\alpha_{i}}{\prod_{i=1}^{n}k_{i}}}
\end{equation*}
be a positive constant. Since $f_{i}(t,0)=0$ for $i=2,\ldots,n$, there exists another constant $r_{2}$:
$$0<r_{2}<{\rm min}\left\lbrace r_{1},L_{1}^{\frac{\prod_{i=1}^{n}k_{i}}{\prod_{i=1}^{n}k_{i}-\prod_{i=1}^{n}
\alpha_{i}}}\right\rbrace $$
such that
\begin{equation}\label{Equation3.4}
f_{i}\left( t,v_{i+1}(t)\right) \le r_{1}^{k_{i}},\quad 0\le v_{i+1}\le r_{2},
\end{equation}
for any $t\in[0,1]$ and $i=2,\ldots,n$.
For $v_{1}\in K\cap\partial\Omega_{r_{2}}$, it follows from Lemma \ref{upper bound} and \eqref{Equation3.4} that 
\begin{equation*}
v_{i}(t)=T_{i}(v_{i+1})(t)<r_{1},\quad i=2,\ldots,n
\end{equation*}
which shows that for any $v_{1}\in  K\cap\partial\Omega_{r_{2}}$, we have $v_{i}\in(0,r_{1})$, for all $i=1,\ldots,n$.
Then by Lemma \ref{lower bound} and \eqref{Equation3.3}, we get
\begin{equation*}
T_{i}(v_{i+1})(\frac{1}{4})\ge \Gamma_{i}(\underline{f}_{i}^{0}-\varepsilon_{1})^{\frac{1}{k_{i}}}(\frac{1}{4})^{\frac{\alpha_{i}}{k_{i}}}||v_{i+1}||^{\frac{\alpha_{i}}{k_{i}}}, \quad i=1,\ldots,n.
\end{equation*}
This suggests that for any  $v_{1}\in  K\cap\partial\Omega_{r_{2}}$, we have
\begin{equation*}
	\begin{split}
		||Tv_{1}||=&\sup_{t\in [0,1]}\left| T_{1}T_{2}\cdots T_{n}(v_{1})(t)\right| \\
		\ge&T_{1}T_{2}\cdots T_{n}(v_{1})(\frac{1}{4})\\
		\ge&\Gamma_{1}(\underline{f}_{1}^{0}-\varepsilon_{1})^{\frac{1}{k_{1}}}(\frac{1}{4})^{\frac{\alpha_{1}}{k_{1}}}||T_{2}(v_{3})||^{\frac{\alpha_{1}}{k_{1}}}\\
		\ge&\Gamma_{1}(\underline{f}_{1}^{0}-\varepsilon_{1})^{\frac{1}{k_{1}}}(\frac{1}{4})^{\frac{\alpha_{1}}{k_{1}}}|T_{2}(v_{3})(\frac{1}{4})|^{\frac{\alpha_{1}}{k_{1}}}\\
		\ge&\Gamma_{1}\Gamma_{2}^{\frac{\alpha_{1}}{k_{1}}}
		(\underline{f}_{1}^{0}-\varepsilon_{1})^{\frac{1}{k_{1}}}
		( \underline{f}_{2}^{0}-\varepsilon_{1} )^{\frac{\alpha_{1}}{k_{1}k_{2}}}
		(\frac{1}{4})^{\frac{\alpha_{1}}{k_{1}}+\frac{\alpha_{2}\alpha_{2}}{k_{1}k_{2}}}||T_{3}(v_{4})||^{\frac{\alpha_{1}\alpha_{2}}{k_{1}k_{2}}}\\
		&\qquad\vdots\\
		\ge&L_{1}||v_{1}||^{\frac{\prod_{i=1}^{n}\alpha_{i}}{\prod_{i=1}^{n}k_{i}}}.
	\end{split}
\end{equation*}
Notice that $||v_{1}||=r_{2}< L_{1}^{\frac{\prod_{i=1}^{n}k_{i}}{\prod_{i=1}^{n}k_{i}-\prod_{i=1}^{n}\alpha_{i}}} $ and $\prod_{i=1}^{n}\alpha_{i}<\prod_{i=1}^{n}k_{i}$, then
\begin{equation*}
	\frac{L_{1}||v_{1}||^{\frac{\prod_{i=1}^{n}\alpha_{i}}{\prod_{i=1}^{n}k_{i}}}}{||v_{1}||}=\frac{L_{1}}{||v_{1}||^{\frac{\prod_{i=1}^{n}k_{i}-\prod_{i=1}^{n}\alpha_{i}}{\prod_{i=1}^{n}k_{i}}}}>1,
\end{equation*}
which implies that 
\begin{equation}\label{Th1.1ge}
	||Tv_{1}||>||v_{1}||,\quad v_{1}\in K\cap \partial\Omega_{r_{2}}.
\end{equation}

On the other hand, it can be obtained from  $\overline{f}_{i}^{\infty}\in (0,+\infty)$ that for any given $\varepsilon_{2}>0$, there exists a constant $R_{1}>1$ such that for any $t\in [0,1]$ and $i=1,\ldots,n$, 
\begin{equation}\label{Equation3.12}
f_{i}\left( t,v_{i+1}(t)\right)  \le (\overline{f}_{i}^{\infty}+\varepsilon_{2})v_{i+1}^{\beta_{i}}, \quad v_{i+1}\ge R_{1}.
\end{equation}
Furthermore, by the continuity of $f_{i}\ (i=1,\ldots,n)$, there exist constants $M_{i}(R_{1})>0\ (i=1,\ldots,n)$ such that for any $(t,v_{i+1}(t))\in [0,1]\times[0,R_{1}]$,
\begin{equation}\label{Equation3.13}
f_{i}\left( t,v_{i+1}(t)\right)  \le M_{i}(R_{1}), \quad i=1,\ldots,n.
\end{equation}
Combining \eqref{Equation3.12} with \eqref{Equation3.13}, we have for any $(t,v_{i+1}(t))\in [0,1]\times \left[ 0,+\infty\right)$,
\begin{equation}\label{Equation3.14}
f_{i}\left( t,v_{i+1}(t)\right)  \le M_{i}(R_{1})+ (\overline{f}_{i}^{\infty}+\varepsilon_{2})v_{i+1}^{\beta_{i}}, \quad i=1,\ldots,n.
\end{equation}
Then from the Lemma \ref{upper bound} and \eqref{Equation3.14}, we have for any $t\in [0,1]$,
\begin{equation*}
T_{i}(v_{i+1})(t) \le \left[ M_{i}(R_{1})+(\overline{f}_{i}^{\infty}+\varepsilon_{2})||v_{i+1}||^{\beta_{i}}\right] ^{\frac{1}{k_{i}}},\quad i=1,\ldots,n.
\end{equation*}
Let
\begin{equation*}
   \begin{aligned}
	&H:=M_{1}(R_{1})^{\frac{1}{k_{1}}}+\cdots+(\overline{f}_{1}^{\infty}+\varepsilon_{2})^{\frac{1}{k_{1}}}\cdots(\overline{f}_{n-1}^{\infty}+\varepsilon_{2})^{\frac{\prod_{i=1}^{n-2}\beta_{i}}{\prod_{i=1}^{n-1}k_{i}}}M_{n}(R_{1})^{\frac{\prod_{i=1}^{n-1}\beta_{i}}{\prod_{i=1}^{n}k_{i}}},\\
    &L_{2}:=(\overline{f}_{1}^{\infty}+\varepsilon_{2})^{\frac{1}{k_{1}}}\cdots(\overline{f}_{n}^{\infty}+\varepsilon_{2})^{\frac{\prod_{i=1}^{n-1}\beta_{i}}{\prod_{i=1}^{n}k_{i}}}.
   \end{aligned}
\end{equation*} 
Thus, there exists a large constant $R_{2}$:
$$R_{2}>{\rm max}\left\lbrace R_{1},2H,(2L_{2})^{\frac{\prod_{i=1}^{n}k_{i}}{\prod_{i=1}^{n}k_{i}-\prod_{i=1}^{n}\beta_{i}}} \right\rbrace $$
such that for any $v_{1}\in K\cap \partial\Omega_{R_{2}}$ and $t\in [0,1]$,
\begin{equation*}
	\begin{split}
		Tv_{1}(t)=&T_{1}T_{2}\cdots T_{n}(v_{1})(t)\\
		\le&\left[ M_{1}(R_{1})+(\overline{f}_{1}^{\infty}+\varepsilon_{2})||T_{2}(v_{3})||^{\beta_{1}}\right] ^{\frac{1}{k_{1}}}\\
		\le&M_{1}(R_{1})^{\frac{1}{k_{1}}}+\left[ (\overline{f}_{1}^{\infty}+\varepsilon_{2})||T_{2}(v_{3})||^{\beta_{1}}\right] ^{\frac{1}{k_{1}}}\\	
		\le&M_{1}(R_{1})^{\frac{1}{k_{1}}}+
		(\overline{f}_{1}^{\infty}+\varepsilon_{2})^{\frac{1}{k_{1}}}M_{2}(R_{1})^{\frac{\beta_{1}}{k_{1}k_{2}}}+
		(\overline{f}_{1}^{\infty}+\varepsilon_{2})^{\frac{1}{k_{1}}}(\overline{f}_{2}^{\infty}+\varepsilon_{2})^{\frac{\beta_{1}}{k_{1}k_{2}}}||v_{3}||^{\frac{\beta_{1}\beta_{2}}{k_{1}k_{2}}}\\
	    &\qquad\vdots\\
	    \le&H+L_{2}||v_{1}||^{\frac{\prod_{i=1}^{n}\beta_{i}}{\prod_{i=1}^{n}k_{i}}}.
	\end{split}
\end{equation*}
Since $\prod_{i=1}^{n}\beta_{i}<\prod_{i=1}^{n}k_{i}$, we get that for any $v_{1}\in K\cap\partial\Omega_{R_{2}}$, 
\begin{equation*}
	\frac{H+L_{2}||v_{1}||^{\frac{\prod_{i=1}^{n}\beta_{i}}{\prod_{i=1}^{n}k_{i}}}}{||v_{1}||}=\frac{H}{||v_{1}||}+\frac{L_{2}}{||v_{1}||^{\frac{\prod_{i=1}^{n}k_{i}-\prod_{i=1}^{n}\beta_{i}}{\prod_{i=1}^{n}k_{i}}}}<1,
\end{equation*}
which implies that 
\begin{equation}\label{Th1.1le}
	||Tv_{1}||<||v_{1}||,\quad v_{1}\in K\cap\partial\Omega_{R_{2}}.
\end{equation}

Therefore, combining with \eqref{Th1.1ge} and \eqref{Th1.1le}, it follows from Theorem \ref{Lemma1.1} that $T$ has at least one fixed point in $K\cap \left(  \overline{\Omega}_{R_{2}} \setminus \Omega_{r_{2}} \right) $.

(b). By the assumption of $\overline{f}_{i}^{0}\in (0,+\infty)$, for any given $\eta_{1}>0$, there exists a positive constant $r_{3}<1$ such that for any $t\in[0,1]$ and $i=1,\ldots,n$,
\begin{equation}\label{Equation3.20}
f_{i}\left( t,v_{i+1}(t)\right)  \le (\overline{f}_{i}^{0}+\eta_{1})v_{i+1}^{\alpha_{i}}, \quad v_{i+1}\in [0,r_{3}].
\end{equation}
Let 
\begin{equation*}
	L_{3}:=(\overline{f}_{1}^{0}+\eta_{1})^{\frac{1}{k_{1}}}
(\overline{f}_{2}^{0}+\eta_{1})^{\frac{\alpha_{1}}{k_{1}k_{2}}}\cdots
(\overline{f}_{n}^{0}+\eta_{1})^{\frac{\prod_{i=1}^{n-1}\alpha_{i}}{\prod_{i=1}^{n}k_{i}}}.
\end{equation*}
Since $\overline{f}_{i}^{0}\in (0,+\infty)$, there exists another constant $r_{4}$:
$$0<r_{4}<\min\left\lbrace r_{3},L_{3}^{
	\frac{\prod_{i=1}^{n}k_{i}}	
	{\prod_{i=1}^{n}k_{i}-\prod_{i=1}^{n}\alpha_{i}}
}\right\rbrace$$
such that for $i=2,\ldots,n$,
\begin{equation}\label{Equation3.21}
f_{i}(t,v_{i+1}(t))\le r_{3}^{k_{i}},\quad (t,v_{i+1}(t))\in [0,1]\times[0,r_{4}].
\end{equation}
Then for any $v_{1}\in K\cap \partial\Omega_{r_{4}}$, it follows from Lemma \ref{upper bound} and \eqref{Equation3.21} that 
\begin{equation*}
v_{i}(t)=T_{i}(v_{i+1})(t)\le r_{3},\quad (t,v_{i+1}(t))\in [0,1]\times[0,r_{4}],\quad i=2,\ldots,n.
\end{equation*}
Thus, By Lemma \ref{upper bound} and \eqref{Equation3.20}, we get 
\begin{equation*}
T_{i}(v_{i+1})(t)\le  \left[ (\overline{f}_{i}^{0}+\eta_{1})||v_{i+1}||^{\alpha_{i}}\right] ^{\frac{1}{k_{i}}}, \quad i=1,\ldots,n,
\end{equation*}
for any $t\in [0,1]$.
For $v_{1}\in K\cap \partial\Omega_{r_{4}}$, we have 
\begin{equation*}
	\begin{split}
		||Tv_{1}||=&\sup_{t\in [0,1]}\left| T_{1}T_{2}\cdots T_{n}(v_{1})(t)\right| \\
		\le& \left[ (\overline{f}_{1}^{0}+\eta_{1})||T_{2}(v_{3})||^{\alpha_{1}}\right] ^{\frac{1}{k_{1}}}\\
		\le&(\overline{f}_{1}^{0}+\eta_{1})^{\frac{1}{k_{1}}}
		\left[(\overline{f}_{2}^{0}+\eta_{1})||T_{3}(v_{4})||^{\alpha_{2}}\right] ^{\frac{\alpha_{1}}{k_{1}k_{2}}}\\
		\le&(\overline{f}_{1}^{0}+\eta_{1})^{\frac{1}{k_{1}}}
		(\overline{f}_{2}^{0}+\eta_{1})^{\frac{\alpha_{1}}{k_{1}k_{2}}}
		\left[ (\overline{f}_{3}^{0}+\eta_{1})||T_4(v_{5})||^{\alpha_{3}}\right] ^{\frac{\alpha_{1}\alpha_{2}}{k_{1}k_{2}k_{3}}}\\
		&\qquad\vdots\\
		\le&L_{3}||v_{1}||^{\frac{\prod_{i=1}^{n}\alpha_{i}}{\prod_{i=1}^{n}k_{i}}}.
	\end{split}
\end{equation*}
Recalling that $\prod_{i=1}^{n}\alpha_{i}>\prod_{i=1}^{n}k_{i}$, then 
\begin{equation*}
   \frac
{ L_{3}||v_{1}||^{\frac{\prod_{i=1}^{n}\alpha_{i}}{\prod_{i=1}^{n}k_{i}}} }
{||v_{1}||}
   =\frac
{L_{3}}
{||v_{1}||^{\frac{\prod_{i=1}^{n}k_{i}-\prod_{i=1}^{n}\alpha_{i}}{\prod_{i=1}^{n}k_{i}}}}
   <1,
	\end{equation*}
which implies that 
\begin{equation}\label{Th1.2le}
	||Tv_{1}||<||v_{1}||,\quad v_{1}\in K\cap\partial\Omega_{r_{4}}.
\end{equation}

On the other hand, it follows from $\underline{f}_{i}^{\infty}\in (0,+\infty)$ that for any given  $\eta_{2}\in (0,\min\{\underline{f}_{i}^{\infty},i=1,\ldots,n\}) $, there exists a constant $R_{3}>1$ such that 
\begin{equation}\label{Equation3.28}
f_{i}\left( t,v_{i+1}(t)\right)  \ge (\underline{f}_{i}^{\infty}-\eta_{2})v_{i+1}^{\beta_{i}}, \quad v_{i+1}\ge R_{3},
\end{equation}
for any $t\in[0,1]$ and $i=1,\ldots,n$.
Let
\begin{equation*}
		L_{4}:=\Gamma_{1}\cdots\Gamma_{n}^{\frac{\prod_{i=1}^{n-1}\beta_{i}}{\prod_{i=1}^{n-1}k_{i}}}
		(\underline{f}_{1}^{\infty}-\eta_{2})^{\frac{1}{k_{1}}}
		\cdots
		(\underline{f}_{n}^{\infty}-\eta_{2})^{\frac{\prod_{i=1}^{n-1}\beta_{i}}{\prod_{i=1}^{n}k_{i}}}
		(\frac{1}{4})^{\frac{\beta_{1}}{k_{1}}+\cdots+\frac{\prod_{i=1}^{n}\beta_{i}}{\prod_{i=1}^{n}k_{i}}}.
\end{equation*}
There exists another constant $R_{4}$:
\begin{equation}\label{Equation3.29}
	R_{4}> \left\lbrace 4R_{3},L_{4}^{\frac{\prod_{i=1}^{n}k_{i}}{\prod_{i=1}^{n}k_{i}-\prod_{i=1}^{n}\beta_{i}}},L_5\right\rbrace 
\end{equation}
such that for any $v_{1}\in K\cap \partial\Omega_{R_{4}}$, we have 
\begin{equation}\label{Equation3.30}
	\min_{\frac{1}{4}\le t\le\frac{3}{4}}v_{1}(t)\ge\frac{1}{4}||v_{1}||=\frac{1}{4}R_{4}>R_{3},
\end{equation} 
where $$L_5:=
\max_{l \in \{2,\ldots,n\}}
 \left( \frac{4R_{3}}
{\Gamma_{l}
\cdots\Gamma_{n}^{\frac{\prod_{i=l}^{n-1}\beta_{i}}{\prod_{i=l}^{n-1}k_{i}}}
(\underline{f}_{l}^{\infty}-\eta_{2})^{\frac{1}{k_{l}}}
\cdots(\underline{f}_{n}^{\infty}-\eta_{2})^{\frac{\prod_{i=l}^{n-1}\beta_{i}}{\prod_{i=l}^{n}k_{i}}}
(\frac{1}{4})^{\frac{\beta_{l}}{k_{l}}+\cdots+\frac{\prod_{i=l}^{n}\beta_{l}}{\prod_{i=l}^{n}k_{l}}}   }
\right) ^{\frac{\prod_{i=l}^{n}k_{i}}{\prod_{i=l}^{n}\beta_{i}}}.$$
Here in $L_5$, when $l=n$ we set $\frac{\prod_{i=l}^{n-1}\beta_{i}}{\prod_{i=l}^{n-1}k_{i}}=1$, so that the terms $\frac{\prod_{i=l}^{n-1}\beta_{i}}{\prod_{i=l}^{n-1}k_{i}}$ make sense for all $i=2, \ldots, n$.
Combining \eqref{Equation3.28} and \eqref{Equation3.29}, it follows from Lemma \ref{lower bound}  that for any $v_{1}\in K\cap \partial\Omega_{R_{4}}$,
\begin{equation}\label{Equation3.31}
	\begin{aligned}
		&||v_{n}||\ge v_{n}(\frac{1}{4})=T_{n}(v_{1})(\frac{1}{4})\ge \Gamma_{n}(\underline{f}_{n}^{\infty}-\eta_{2})^{\frac{1}{k_{n}}}(\frac{1}{4})^{\frac{\beta_{n}}{k_{n}}}||v_{1}||^{\frac{\beta_{n}}{k_{n}}}>4R_{3},\\
		&||v_{n-1}||\ge v_{n-1}(\frac{1}{4})=T_{n-1}(v_{n})(\frac{1}{4})\ge \Gamma_{n-1}(\underline{f}_{n-1}^{\infty}-\eta_{2})^{\frac{1}{k_{n-1}}}(\frac{1}{4})^{\frac{\beta_{n-1}}{k_{n-1}}}||v_{n}||^{\frac{\beta_{n-1}}{k_{n-1}}}>4R_{3},\\
		&\qquad\vdots\\
		&||v_{2}||\ge v_{2}(\frac{1}{4})=T_{2}(v_{3})(\frac{1}{4})\ge \Gamma_{2}(\underline{f}_{2}^{\infty}-\eta_{2})^{\frac{1}{k_{2}}}(\frac{1}{4})^{\frac{\beta_{2}}{k_{2}}}||v_{3}||^{\frac{\beta_{2}}{k_{2}}}>4R_{3}.
	\end{aligned}
\end{equation}
From \eqref{Equation3.30} and \eqref{Equation3.31}, we get that for any $v_{1}\in K\cap \partial\Omega_{R_{4}}$,
\begin{equation*}
	\min_{\frac{1}{4}\le t\le \frac{3}{4}} v_{i}(t) \ge \frac{1}{4} ||v_{i}|| \ge R_{3},\quad i=1,\ldots,n. 
\end{equation*}
Then by Lemma \ref{lower bound}, we deduce that 
\begin{equation*}
	\begin{split}
		Tv_{1}(\frac{1}{4})=&T_{1}T_{2}\cdots T_{n}(v_{1})(\frac{1}{4})\\
		\ge&\Gamma_{1}(\underline{f}_{1}^{\infty}-\eta_{2})^{\frac{1}{k_{1}}}(\frac{1}{4})^{\frac{\beta_{1}}{k_{1}}}||T_{2}(v_{3})||^{\frac{\beta_{1}}{k_{1}}}\\
		\ge&\Gamma_{1}(\underline{f}_{1}^{\infty}-\eta_{2})^{\frac{1}{k_{1}}}(\frac{1}{4})^{\frac{\beta_{1}}{k_{1}}}
		\left(\Gamma_{2}(\underline{f}_{2}^{\infty}-\eta_{2})^{\frac{1}{k_{2}}}(\frac{1}{4})^{\frac{\beta_{2}}{k_{2}}}||T_3(v_{4})||^{\frac{\beta_{2}}{k_{2}}} \right) 
		^{\frac{\beta_{1}}{k_{1}}}\\
		&\qquad\vdots\\
		\ge&L_{4}||v_{1}||^{\frac{\prod_{i=1}^{n}\beta_{i}}{\prod_{i=1}^{n}k_{i}}}.
	\end{split}
\end{equation*}
Since $\prod_{i=1}^{n}\beta_{i}>\prod_{i=1}^{n}k_{i}$, it follows from 
\begin{equation*}
	\frac{L_{4}||v_{1}||^{\frac{\prod_{i=1}^{n}\beta_{i}}{\prod_{i=1}^{n}k_{i}}}}{||v_{1}||}=\frac{L_{4}}{||v_{1}||^{\frac{\prod_{i=1}^{n}k_{i}-\prod_{i=1}^{n}\beta_{i}}{\prod_{i=1}^{n}k_{i}}}}>1
\end{equation*}
that 
\begin{equation}\label{Th1.2ge}
	||Tv_{1}||>||v_{1}||,\quad v_{1}\in K\cap \partial\Omega_{R_{4}}.
\end{equation}

Therefore, from Theorem \ref{Lemma1.1} combining  \eqref{Th1.2le} and \eqref{Th1.2ge}, we obtain that $T$ has at least one fixed point in $K\cap \left(  \overline{\Omega}_{R_{4}} \setminus \Omega_{r_{4}} \right) $.
\end{proof}

\subsection{Multiplicity}\label{Section 3.2}

In Section \ref{Section 3.1}, applying the fixed-point theorem in Theorem \ref{Lemma1.1},  we achieve the existence result in a cone with different combinations of asymptotic growth condition and relations of $\alpha_{i},\beta_{i}$ and $k_{i}$. In order to obtain the multiplicity of nontrivial radial convex solutions of \eqref{EQUATION1}, we recombine the conditions in Theorem \ref{Th1.1} and find two kinds of ``intermediate state" as in \eqref{Th1.3le} and \eqref{Th1.4ge}.

\begin{proof}[Proof of Theorem \ref{Th1.2}] 
(c). As we assumed, for any $(t,v_{1}(t))\in [0,1]\times[0,\frac{r_{0}}{4}]$, we have 
\begin{equation*}
	||v_{1}||\le 4 \min_{\frac{1}{4}\le t\le \frac{3}{4}}v_{1}(t)\le r_{0}.
\end{equation*}
Then for $v_{1}\in K\cap \partial\Omega_{r_{0}}$, by the definition of $G_{n}$, we have 
\begin{equation*}
	\begin{split}
		v_{n}(t)=T_{n}(v_{1})(t)
		\le& \int_0^1\left( \frac{k_{n}}{\tau^{N-k_{n}}}\int_0^\tau \frac{s^{N-1}}{C_{N-1}^{k_{n}-1}}f_{n}\left( s,v_{1}(s)\right)\,{\rm d}s \right)^{\frac{1}{k_{n}} }\,{\rm d}\tau\\
		\le& \int_0^1\left( \frac{k_{n}}{\tau^{N-k_{n}}}\int_0^\tau \frac{s^{N-1}}{C_{N-1}^{k_{n}-1}}   G_{n}
		\right)^{\frac{1}{k_{n}} }\,{\rm d}\tau\\
		=&\frac{1}{2}\left(\frac{k_{n}G_{n}}{NC_{N-1}^{k_{n}-1}} \right) ^{\frac{1}{k_{n}} }\\
		<& G_{n}^{\frac{1}{k_{n}} },
	\end{split}
\end{equation*}
for any $t\in[0,1]$.
Similarly, for $v_{1}\in K\cap \partial\Omega_{r_{0}}$,  we have $v_{i}(t) \le G_{i}^{\frac{1}{k_{i}}}\ (i=2,\ldots,n-1)$, $\forall t\in[0,1]$.
Therefore, by the definition of $G_{1}$, we have
\begin{equation*}
	\begin{split}
		Tv_{1}(t)=&T_{1}T_{2}\cdots T_{n}(v_{1})(t)\\
		=& \int_t^1\left( \frac{k_{1}}{\tau^{N-k_{1}}}\int_0^\tau \frac{s^{N-1}}{C_{N-1}^{k_{1}-1}}f_{1}\left( s,v_{2}(s)\right)\,{\rm d}s \right)^{\frac{1}{k_{1}} }\,{\rm d}\tau\\
		\le& \int_0^1\left( \frac{k_{1}}{\tau^{N-k_{1}}}\int_0^\tau \frac{s^{N-1}}{C_{N-1}^{k_{1}-1}}G_{1}\,{\rm d}s \right)^{\frac{1}{k_{1}} }\,{\rm d}\tau\\
		=&\frac{1}{2}\left(\frac{k_{1}G_{1}}{NC_{N-1}^{k_{1}-1}} \right) ^{\frac{1}{k_{1}} }\\
		<&G_{1}^{\frac{1}{k_{1}}},\quad\forall t\in[0,1],
	\end{split}	
\end{equation*}
which implies that 
\begin{equation}\label{Th1.3le}
	||Tv_{1}||<||v_{1}||,\quad v_{1}\in K\cap \partial\Omega_{r_{0}}.
\end{equation}

Since $\prod_{i=1}^{n}\alpha_{i}<\prod_{i=1}^{n}k_{i}$, $\prod_{i=1}^{n}\beta_{i}>\prod_{i=1}^{n}k_{i}$, it follows from Theorem \ref{Th1.1} that there exist sufficient small constant $r_{2}\in (0,r_{0})$ and sufficient large constant $R_{4}>r_{0}$ such that 
\begin{equation}\label{Th1.3ge}
	||Tv_{1}||\ge ||v_{1}||,\quad v_{1}\in K\cap \partial\Omega_{r_{2}} \quad {\rm and} \quad 
    ||Tv_{1}||\ge ||v_{1}||,\quad v_{1}\in K\cap \partial\Omega_{R_{4}}.
\end{equation}

Note that for any $v_{1}\in K\cap \partial\Omega_{r_{0}}$, $||Tv_{1}||\neq||v_{1}||$, which shows that $T$ has no fixed point in $K\cap \partial\Omega_{r_{0}}$. Combining \eqref{Th1.3le} and \eqref{Th1.3ge}, it follows from Theorem \ref{Lemma1.1} that there exist at least two fixed points of $T$ in $K\cap \left( \overline{\Omega}_{r_{0}}\setminus \Omega_{r_{2}}\right) $ and $K\cap \left( \overline{\Omega}_{R_{4}}\setminus \Omega_{r_{0}}\right) $ respectively.
	
(d). For $v_{1}\in K\cap \partial\Omega_{R_{0}}$, by the definition of $\tilde{G}_{n}$, we have 
	\begin{equation*}
		\begin{split}
		v_{n}(t)=T_{n}(v_{1})(t)
		\le& \int_0^1\left( \frac{k_{n}}{\tau^{N-k_{n}}}\int_0^\tau \frac{s^{N-1}}{C_{N-1}^{k_{n}-1}}f_{n}\left( s,v_{1}(s)\right)\,{\rm d}s \right)^{\frac{1}{k_{n}} }\,{\rm d}\tau\\
		\le& \int_0^1\left( \frac{k_{n}}{\tau^{N-k_{n}}}\int_0^\tau \frac{s^{N-1}}{C_{N-1}^{k_{n}-1}}  \tilde{G}_{n}\,{\rm d}s
		\right)^{\frac{1}{k_{n}} }\,{\rm d}\tau\\
		=&\frac{1}{2}\left(\frac{k_{n}\tilde{G}_{n}}{NC_{N-1}^{k_{n}-1}} \right) ^{\frac{1}{k_{n}} }\\
		<& \tilde{G}_{n}^{\frac{1}{k_{n}} },\quad\forall t\in[0,1].
	\end{split}	
	\end{equation*}
Besides, by the definition of $E_n$, we get that for any $v_{1}\in K\cap \partial\Omega_{R_{0}}$,
	\begin{equation*}
	\begin{split}
		v_{n}(\frac{1}{4})=T_{n}(v_{1})(\frac{1}{4})
		=& \int_{\frac{1}{4}}^1\left( \frac{k_{n}}{\tau^{N-k_{n}}}\int_0^\tau \frac{s^{N-1}}{C_{N-1}^{k_{n}-1}}f_{n}\left( s,v_{1}(s)\right)\,{\rm d}s \right)^{\frac{1}{k_{n}} }\,{\rm d}\tau\\
		\ge& \int_{\frac{1}{4}}^{\frac{3}{4}}\left( \frac{k_{n}}{\tau^{N-k_{n}}}\int_{\frac{1}{4}}^\tau \frac{s^{N-1}}{C_{N-1}^{k_{n}-1}}  E_{n}\,{\rm d}s
		\right)^{\frac{1}{k_{n}} }\,{\rm d}\tau\\
		=& \Gamma_{n} E_{n}^{\frac{1}{k_{n}} },
	\end{split}	
\end{equation*}
then
\begin{equation*}
	\min_{\frac{1}{4}\le t\le \frac{3}{4}} v_{n}(t) \ge \frac{1}{4}||v_{n}||\ge \frac{1}{4}v_{n}(\frac{1}{4}) \ge \frac{1}{4}\Gamma_{n} E_{n}^{\frac{1}{k_{n}} }.
\end{equation*}
Thus for any $t\in[\frac{1}{4},\frac{3}{4}]$, we have $\frac{1}{4}\Gamma_{n} E_{n}^{\frac{1}{k_{n}} }
\le v_{n}(t) \le  \tilde{G}_{n}^{\frac{1}{k_{n}} }$.
Repeating the above steps, we have $\frac{1}{4}\Gamma_{i} E_{i}^{\frac{1}{k_{i}} }
\le v_{i}(t) \le  \tilde{G}_{i}^{\frac{1}{k_{i}} }$, for any  $t\in[\frac{1}{4},\frac{3}{4}],\ (i=2,\ldots,n)$.
It follows from the assumption of $E_1$ that for any  $v_{1}\in K\cap \partial\Omega_{R_{0}}$,
\begin{equation*}
	\begin{split}
		Tv_{1}(\frac{1}{4})=&T_{1}T_{2}\cdots T_{n}(v_{1})(\frac{1}{4})\\
		=& \int_{\frac{1}{4}}^1\left( \frac{k_{1}}{\tau^{N-k_{1}}}\int_0^\tau \frac{s^{N-1}}{C_{N-1}^{k_{1}-1}}f_{1}\left( s,T_{2}(v_{3})(s)\right)\,{\rm d}s \right)^{\frac{1}{k_{1}} }\,{\rm d}\tau\\
		\ge& \int_{\frac{1}{4}}^{\frac{3}{4}}\left( \frac{k_{1}}{\tau^{N-k_{1}}}\int_0^\tau \frac{s^{N-1}}{C_{N-1}^{k_{1}-1}}  E_{1}\,{\rm d}s
		\right)^{\frac{1}{k_{1}} }\,{\rm d}\tau\\
		\ge& \Gamma_{1} E_{1}^{\frac{1}{k_{1}}}
		>R_{0},
	\end{split}
	\end{equation*}
which deduce that 
\begin{equation}\label{Th1.4ge}
	||Tv_{1}||>||v_{1}||, \quad v_{1}\in K\cap \partial \Omega_{R_{0}}.
\end{equation}

Moreover, since $\prod_{i=1}^{n}\alpha_{i}>\prod_{i=1}^{n}k_{i},\prod_{i=1}^{n}\beta_{i}<\prod_{i=1}^{n}k_{i}$, we know from Theorem \ref{Th1.1}  that there exist sufficient small constant $r_{4}\in (0,R_{0})$ and sufficient large constant $R_{2}>R_{0}$ such that 
\begin{equation}\label{Th1.4le}
	||Tv_{1}||\le ||v_{1}||,\quad v_{1}\in K\cap \partial\Omega_{R_{2}} \quad {\rm and} \quad 
	||Tv_{1}||\le ||v_{1}||,\quad v_{1}\in K\cap \partial\Omega_{r_{4}}.
\end{equation}

Note that for $v_{1}\in K\cap \partial \Omega_{R_{0}}$, the norm of $Tv_{1}$ is strictly greater than $||v_{1}||$, which shows that fixed point of $T$ can not exists on $K\cap\partial\Omega_{R_{0}}$. Thus, basing on \eqref{Th1.4ge} and \eqref{Th1.4le}, it follows from Theorem \ref{Lemma1.1} that there exist at least two fixed points of $T$ in $K\cap \left( \overline{\Omega}_{R_{0}}\setminus \Omega_{r_{4}}\right) $ and $K\cap \left( \overline{\Omega}_{R_{2}}\setminus \Omega_{R_{0}}\right) $ respectively.
\end{proof}

\vspace{3mm}


\section{Uniqueness and nonexistence}\label{Section 4}
In this section, we study the uniqueness and nonexistence results for a special case of the system \eqref{EQUATION1} where the nonlinearities are power functions with respect to $u$.
\subsection{Uniqueness}\label{Section 4.1}
In \cite{HW}, the authors gave a proof of uniqueness and approximation by iterations of the solution to a general Dirichlet problem of Monge-Amp\`{e}re equation. Here, we will use their method to prove Theorem \ref{Th1.3}. We first introduce the definition of $u_{0}$-sublinear operator and a corresponding existence result.

\begin{Definition}\label{Def4.1}
Let $P$ be a cone from a Banach space $Y$. With some $u_{0} \in P$ positive, $A:P \to P$ is called $u_{0}$-sublinear if
\begin{enumerate}[(i)]
	\item for any $x>0$, there exist positive constants $\theta_1$ and $\theta_2$ which depend on $x$, such that 
	$$\theta_1 u_0 \le Ax \le \theta_2 u_0;$$
	\item for any $\theta_1 u_0 \le x \le \theta_2 u_0$ and $0 < \xi < 1$, there always exists some $\eta > 0$ such that 
	$$A(\xi x) \ge (1 + \eta)\xi Ax.$$
\end{enumerate}
\end{Definition}

\begin{Lemma}\label{Lemma4.1}
An increasing and $u_0$-sublinear operator $A$ can have at most one positive fixed-point.	
\end{Lemma}
The proof can be found in \cite{HW}, we omit it here.

\begin{proof}[Proof of Theorem \ref{Th1.3}]
Let $X:=C[0,1]$ and cone $P:=\{v\in X:v(t)\ge 0, t\in [0,1]\}$. 
It is easy to see that $K\subset P$, where $K$ is defined in \eqref{cone}. 
We define $T_{i}\ (i=1,\ldots,n)$ and composite operator $T=T_{1}T_{2}\cdots T_{n}$ as in Section \ref{Section 2}. 
The existence of nontrivial radial convex solutions to system \eqref{EQUATION2} is obtained in Theorem \ref{Th1.1} and therefore  investigate $T$ has at most one fixed-point in $K$ is enough.
By Lemma \ref{Lemma4.1}, it suffices to verify that $T:K\to K$ is an increasing and $u_0$-sublinear for some $u_0$ positive in $C[0,1]$. 
By the definitions of $T_{i} $, it is clear that each $T_i\ (i=1,\ldots,n)$ is a increasing operator, so is the composite operator $T$, then we just need to prove that $T$ satisfies the Definition \ref{Def4.1}.

Firstly, we show that $T$ satisfies the Definition \ref{Def4.1} (i).
\begin{equation*}
	\begin{split}
	 Tv_1(t)=&\int_{t}^{1}\left(\frac{k_1}{\tau^{N-k_1}}\int_{0}^{\tau}\frac{s^{N-1}}{C_{N-1}^{k_{1}-1}}v_{2}^{\gamma_1}(s)\,{\rm d}s \right)^{\frac{1}{k_{1}}} \,{\rm d}\tau\\
	 \le&||v_{2}||^{\frac{\gamma_1}{k_{1}}}
	 \int_{t}^{1}\left(\frac{k_1}{\tau^{N-k_1}}\int_{0}^{\tau}\frac{s^{N-1}}{C_{N-1}^{k_{1}-1}}\,{\rm d}s \right)^{\frac{1}{k_{1}}} \,{\rm d}\tau \\
	 \le&||T_2(v_3)||^{\frac{\gamma_1}{k_{1}}} \left(\frac{k_1}{NC_{N-1}^{k_1-1}} \right)^{\frac{1}{k_1}}\int_{t}^{1}\tau\,{\rm d}\tau\\
	 \le&\left[ 
	 \int_{0}^{1}\left(\frac{k_2}{\tau^{N-k_2}}\int_{0}^{\tau}\frac{s^{N-1}}{C_{N-1}^{k_{2}-1}}v_{3}^{\gamma_2}(s)\,{\rm d}s \right)^{\frac{1}{k_{2}}} \,{\rm d}\tau 
	 \right] ^{\frac{\gamma_1}{k_{1}}}(1-t)\\
	 \le& ||v_3||^{\frac{\gamma_1\gamma_2}{k_1k_2}} (1-t)\\
	 &\qquad\vdots\\
	 \le& ||v_1||^{\frac{\prod_{i=1}^{n}\gamma_i}{\prod_{i=1}^{n} k_i}}(1-t).
	\end{split}
\end{equation*}
Here, let $u_0=1-t$, $t\in[0,1)$ and $\theta_2=||v_1||^{\frac{\prod_{i=1}^{n}\gamma_i}{\prod_{i=1}^{n} k_i}}$, then $Tv_1(t)\le\theta_2u_0$.
Set $$\Gamma:=\left( \frac{9}{32}\right) ^{\frac{\gamma_1}{k_1}+\cdots+\frac{\prod_{i=1}^{n-1} \gamma_i}{\prod_{i=1}^{n-1} k_i}}
\left( \frac{k_1}{4^{\gamma_1}NC_{N-1}^{k_1-1}}\right) ^{\frac{1}{k_1}} 
\cdots
\left( \frac{k_n}{4^{\gamma_n}NC_{N-1}^{k_n-1}}\right) ^{\frac{\prod_{i=1}^{n-1} \gamma_i}{\prod_{i=1}^{n} k_i}}
||v_1||^{\frac{\prod_{i=1}^{n}\gamma_i}{\prod_{i=1}^{n} k_i}}$$
be a positive constant which depends only on $||v_1||$.
Next, let $c\in(\frac{1}{4},\frac{3}{4})$ be a fixed number. Notice that  $Tv_1(t)$ is decreasing with $t$, we have for $t\in[0,c)$,
\begin{equation*}
	\begin{split}
   	   Tv_1(t)\ge& Tv_1(c)=\int_{c}^{1}\left(
   	   \frac{k_1}{\tau^{N-k_1}}\int_{0}^{\tau}\frac{s^{N-1}}{C_{N-1}^{k_{1}-1}}v_{2}^{\gamma_1}(s)\,{\rm d}s 
   	   \right)^{\frac{1}{k_{1}}} \,{\rm d}\tau\\
   	   \ge&\int_{c}^{\frac{3}{4}}\left(\frac{k_1}{\tau^{N-k_1}}\int_{0}^{\tau}\frac{s^{N-1}}{C_{N-1}^{k_{1}-1}}(\frac{1}{4}||v_{2}||)^{\gamma_1}\,{\rm d}s \right)^{\frac{1}{k_{1}}} \,{\rm d}\tau\\
   	   \ge&\left( \frac{k_1}{4^{\gamma_1}NC_{N-1}^{k_1-1}}\right) ^{\frac{1}{k_1}} 
   	   \left( \frac{9}{32}-\frac{1}{2}c^2\right) 
   	   ||T_2(v_3)||^{\frac{\gamma_1}{k_1}}\\
   	    &\qquad\qquad\vdots\\
   	   \ge&\Gamma
   	   \left( \frac{9}{32}-\frac{1}{2}c^2\right)\\
   	   \ge&\Gamma
   	   \left( \frac{9}{32}-\frac{1}{2}c^2\right) 
   	   (1-t),
	\end{split}
\end{equation*}
where we use the fact $\min_{0\le t\le \frac{3}{4}}v(t)\ge \frac{1}{4}||v||$ in the above inequality which follows from \eqref{cone} combining with the concavity of $v_i\ (i=1,\ldots,n)$ and $v_i^{\prime}(0)=0$.
For $t\in[c,1)$, we let 
$$\zeta(\tau):=\left(\frac{1}{\tau^{N-k_1}}\int_{0}^{\tau}s^{N-1}(1-s) ^{\gamma_1}\,{\rm d}s \right)^{\frac{1}{k_{1}}},\ \tau\in [c,1].$$ 
Notice that $\zeta(\tau)\in C[c,1]$ and $\zeta(\tau)>0,\ \tau\in [c,1]$ is well-defined, then $\zeta([c,1])$ is the image of a compact set and so is compact which shows that it is both closed and bounded. So $\zeta([c,1])$ has a positive absolute minimum. Besides, by the concavity of $v_i(t)\ (i=1,\ldots,n)$ and $v_i^{\prime}(0)=v_i(1)=0$, we have 
\begin{equation}\label{Equation4.2}
	v_i(t)\ge v_i(0)(1-t), \quad \forall t\in [0,1].
\end{equation}
Then we have for $t\in[c,1)$,
\begin{equation*}
	\begin{split}
		Tv_1(t)=&\int_{t}^{1}\left(
		\frac{k_1}{\tau^{N-k_1}}\int_{0}^{\tau}\frac{s^{N-1}}{C_{N-1}^{k_{1}-1}}v_{2}^{\gamma_1}(s)\,{\rm d}s 
		\right)^{\frac{1}{k_{1}}} \,{\rm d}\tau\\
		\ge&
		\int_{t}^{1}\left(\frac{k_1}{\tau^{N-k_1}}\int_{0}^{\tau}\frac{s^{N-1}}{C_{N-1}^{k_{1}-1}}[v_2(0)(1-s)] ^{\gamma_1}\,{\rm d}s \right)^{\frac{1}{k_{1}}} \,{\rm d}\tau\\
		=&||v_2||^{\frac{\gamma_1}{k_1}}
		\left( \frac{k_1}{C_{N-1}^{k_{1}-1}}\right)^{\frac{1}{k_1}}
		\int_{t}^{1}\left(\frac{1}{\tau^{N-k_1}}\int_{0}^{\tau}s^{N-1}(1-s) ^{\gamma_1}\,{\rm d}s \right)^{\frac{1}{k_{1}}} \,{\rm d}\tau\\
		\ge&
		||T_2(v_3)||^{\frac{\gamma_1}{k_1}}
		\left( \frac{k_1}{C_{N-1}^{k_{1}-1}}\right)^{\frac{1}{k_1}}
	    \min_{\tau \in [c,1]}\zeta(\tau)
		\int_{t}^{1}\,{\rm d}\tau\\
		=&\left[ \int_{0}^{1}\left(\frac{k_2}{\tau^{N-k_2}}\int_{0}^{\tau}\frac{s^{N-1}}{C_{N-1}^{k_{2}-1}}v_{3}^{\gamma_2}(s)\,{\rm d}s \right)^{\frac{1}{k_{2}}} \,{\rm d}\tau 
		\right] ^{\frac{\gamma_1}{k_1}}
		\left( \frac{k_1}{C_{N-1}^{k_{1}-1}}\right)^{\frac{1}{k_1}}
		\min_{\tau \in [c,1]}\zeta(\tau)
		(1-t)\\
		\ge&
		\left[ \int_{0}^{\frac{3}{4}}\left(\frac{k_2}{\tau^{N-k_2}}\int_{0}^{\tau}\frac{s^{N-1}}{C_{N-1}^{k_{2}-1}}(\frac{1}{4}||v_3||)^{\gamma_2}\,{\rm d}s \right)^{\frac{1}{k_{2}}} \,{\rm d}\tau 
		\right] ^{\frac{\gamma_1}{k_1}}\left( \frac{k_1}{C_{N-1}^{k_{1}-1}}\right)^{\frac{1}{k_1}}
		\min_{\tau \in [c,1]}\zeta(\tau)
		(1-t)\\
		=&
		||T_3(v_4)||^{\frac{\gamma_1\gamma_2}{k_1k_2}}
		\left( \frac{k_1}{C_{N-1}^{k_{1}-1}}\right)^{\frac{1}{k_1}}
		\left( \frac{k_2}{4^{\gamma_2}C_{N-1}^{k_{2}-1}}\right)^{\frac{\gamma_1}{k_1k_2}}
		\left( \frac{9}{32}\right) ^{\frac{\gamma_1}{k_1}}		
		\min_{\tau \in [c,1]}\zeta(\tau)
		(1-t)\\
		&\qquad\qquad\vdots\\
		\ge&\Gamma 4^{\frac{\gamma_1}{k_1}}	
		\min_{\tau \in [c,1]}\zeta(\tau)
		(1-t).
		\end{split}
\end{equation*}
Let $\theta_1={\rm min}\left\lbrace 
\Gamma
\left( \frac{9}{32}-\frac{1}{2}c^2\right),
\Gamma
4^{\frac{\gamma_1}{k_1}}	
\min_{\tau \in [c,1]}\zeta(\tau)
\right\rbrace $, then we have $Tv_1(t)\ge\theta_1u_0$. 

To verify the Definition \ref{Def4.1} (ii), we have 
for any $\theta_1 u_0 \le v_1 \le \theta_2 u_0$ and $\xi\in (0,1)$, $T_1(\xi v_2)=\xi^{\frac{\gamma_1}{k_1}}T_1(v_2),\ T_2(\xi v_3)=\xi^{\frac{\gamma_2}{k_2}}T_2(v_3),\ \ldots,\ T_n(\xi v_1)=\xi^{\frac{\gamma_n}{k_n}}T_n(v_1)$. Notice that $\prod_{i=1}^{n}\gamma_i<\prod_{i=1}^{n}k_i$, then there exists $\eta>0$ such that
\begin{equation*}
	T(\xi v_1)=T_1T_2\cdots T_n(\xi v_1)=T_1T_2\cdots T_{n-1}(\xi^{\frac{\gamma_n}{k_n}}T_n(v_1))=\cdots=\xi^{\frac{\prod_{i=1}^{n}\gamma_i}{\prod_{i=1}^{n}k_i}}Tv_1 \ge(1+\eta)\xi Tv_1.
\end{equation*}

Thus $T$ is a $u_0$-sublinear operator and  $T$ has at most one fixed-point in $K$ by Lemma \ref{Lemma4.1} which shows that the system \eqref{EQUATION2} has a unique nontrivial radial convex solution.
\end{proof}

\begin{remark}\label{Rm4.1}
	Note that we also use the convexity of $u_i=-v_i\ (i=1,\ldots,n)$ in this subsection, namely, the inequality \eqref{Equation4.2}.
\end{remark}

\subsection{Nonexistence}\label{Section 4.2}
In the case of $\prod_{i=1}^{n}\gamma_i=\prod_{i=1}^{n}k_i$, we can get nonexistence result by contradiction.
\begin{proof}[Proof of Theorem \ref{Th1.4}]
Suppose, to the contrary, that $v_0$ is a fixed-point of $T$ in $K$, then  $Tv_0=v_0$. It follows immediately from the definition of $T$ that $v_0$ is a concave function satisfying $v_0(1)=0$ and $v_0(t)>0,t\in [0,1)$.

On the other hand, for any  $v_1\in K$, we have 
\begin{equation*}
	\begin{split}
		||T(v_1)||=&\int_{0}^{1}\left( \frac{k_1}{\tau^{N-k_1}} \int_{0}^{\tau} \frac{s^{N-1}}{C_{N-1}^{k_1-1}}v_2^{\gamma_1}(s)\,{\rm d}s\right)^{\frac{1}{k_1}}\,{\rm d}\tau  \\
		\le&||v_2||^{\frac{\gamma_1}{k_1}}\int_{0}^{1}\left( \frac{k_1}{\tau^{N-k_1}} \int_{0}^{\tau} \frac{s^{N-1}}{C_{N-1}^{k_1-1}}\,{\rm d}s\right)^{\frac{1}{k_1}}\,{\rm d}\tau   \\	
		=& \frac{1}{2}\left( \frac{k_1}{NC_{N-1}^{k_1-1}}\right)^{\frac{1}{k_1}} ||v_2||^{\frac{\gamma_1}{k_1}}\\
		<&||v_2||^{\frac{\gamma_1}{k_1}}\\
	   	 &\qquad\vdots\\
	   	<&||v_1||^{\frac{\prod_{i=1}^{n}\gamma_i}{\prod_{i=1}^{n}k_i}}=||v_1||.
	\end{split}
\end{equation*}
Here, we also use the fact $\frac{1}{2}\left( \frac{k_1}{NC_{N-1}^{k_1-1}}\right)^{\frac{1}{k_1}}<1$ for the same reason in \eqref{Lemma3.2}. Thus if we take $v_1=v_0$ in the above estimate, we have $||v_0||$ is strictly larger than $||Tv_0||$. This contradicts $Tv_0=v_0$
and concludes the proof.
\end{proof}

\begin{remark}
		Due to the fact $\frac{1}{2}\left( \frac{k_1}{NC_{N-1}^{k_1-1}}\right)^{\frac{1}{k_1}}<1$, we have a direct proof of the nonexistence theorem by reduction to absurdity without using the fixed-point theorem in Theorem \ref{Lemma1.1}. Therefore, we can obtain the nonexistence for $\boldsymbol{k}$-admissible solutions of system \eqref{EQUATION2} in the assumption of $\prod_{i=1}^{n}\gamma_i=\prod_{i=1}^{n}k_i$, (not just for the convex solutions of system \eqref{EQUATION2}).
\end{remark}

\vspace{3mm}

\section{Eigenvalue problem}\label{Section 5}
In the previous section, we proved the nonexistence of nontrivial radial convex solution to the power-type system \eqref{EQUATION2} in a unit ball. Then by imposing a suitable condition on positive parameters of eigenvalue problem \eqref{EQUATION3}, we can also get the existence of $\boldsymbol{k}$-admissible solution in a general strictly  $(k-1)$-convex domain. In this section, our main tool is the generalized Krein-Rutman theorem in \cite{Jacobsen}. 

We first recall some basic concepts:

Let $E$ be a Banach space, $M\subset E$ be a cone.
\begin{Definition}\label{partial order}
The cone $M$ introduces a partial order in $E$ by the relation
   \begin{center}
			$u\prec v$ \qquad if and only if\qquad $u-v\in M$.
   \end{center}
\end{Definition}

\begin{Definition}
	Define an operator $A: E \to E$. Then	
\begin{enumerate}[(i)]
	\item $A$ is called positive if $A(M) \subset M$;
	\item $A$ is said to be homogeneous if it is positively homogeneous with degree 1;
	\item $A$ is monotone if it satisfies $x \prec y \Rightarrow A(x) \prec A(y)$;
	\item $A$ is called strong (relative to $M$), if for all $u,v \in Im(A) \cap M \setminus \{ \theta\} $, there
	exist positive constants $\delta$ and $\gamma$ that depend on $u$ and $v$ such that
	$u -\delta v \in M$, $v -\gamma u \in M$.
\end{enumerate}
\end{Definition}

The following is the generalized Krein-Rutman theorem developed by Jacobsen in \cite{Jacobsen}. 
\begin{Lemma}\label{Lemma5.1}
	Let $E$ contain a cone $M$, $A: E \to E$ be a completely continuous operator with $A|_{M}: M \to M$ homogeneous, monotone, and strong. Furthermore, assume that there exists a nonzero element $\omega$,  $ A(\omega)\in Im(A)\cap M$. Then there exists a constant $\lambda_0 > 0$ with the following properties:
	\begin{enumerate}[(i)]
	\item There exists $u \in M \setminus \{ \theta\}$, with $u = \lambda_0 A(u)$;
	\item If $v \in M \setminus \{ \theta\}$ and $\lambda > 0$ such that $v = \lambda A(v)$, then $\lambda = \lambda_0$.
\end{enumerate}
\end{Lemma}
For the convenience of the reader, we also present the existence theorems in \cite{T,Wang-I}.
\begin{Lemma}\label{Lemma5.2}(see \cite{T})
Let $\Omega$ be a uniformly $(k-1)$-convex domain in $\mathbb{R}^{N}$, $k=2,\ldots,N$, $\varphi \in C^{0}(\overline{\Omega}) $ and $\psi \ge 0, \in L^{p}(\Omega)$, for $p>\frac{N}{2k}$. Then there exists a unique admissible weak solution $u\in C^{0}(\overline{\Omega})$ to the  problem
\begin{equation*}
	\left\{
\begin{aligned}
	&S_{k}\left(D^{2}u \right) =\psi,  &{\rm in}\ \Omega,\\
	&u=\varphi,  &{\rm on}\ \partial \Omega.
\end{aligned}
\right.
\end{equation*}	
\end{Lemma}

\begin{Lemma}\label{Lemma5.4} (see \cite{Wang-I})
Assume that $\Omega$ is $(k-1)$-convex, $\varphi$, $ \Omega \in C^{3,1}$, $f\in C^{1,1}(\overline{\Omega})$, and $f\ge f_0>0$. Then there is a unique $k$-admissible solution $u \in C^{3,\alpha}(\overline{\Omega})$ to the Dirichlet problem
\begin{equation*}
	\left\{
	\begin{aligned}
		&S_{k}\left(D^{2}u \right) =f(x),  &{\rm in}\ \Omega,\\
		&u=\varphi,  &{\rm on}\ \partial \Omega.
	\end{aligned}
	\right.
\end{equation*}	
\end{Lemma}

\begin{proof}[Proof of Theorem \ref{Th1.5}]
	Let $X$ be a Banach space $C(\overline{\Omega})$ equipped with the supremum norm. Define a cone $P:=\left\lbrace u\in X:u(x)\le0, \forall x\in \Omega \right\rbrace $. Then by the Definition \ref{partial order}, we notice that the partial order induced by $P$ implies that $u \prec v \Longleftrightarrow u(x) \le v(x), \forall x\in \Omega  $.
	
	For $i=1,\ldots,n$, we define $\overline{T}_i:X\to X$, $\overline{T}_i(u_{i+1})=u_i$, where $u_i$ is the unique admissible weak solution of the problem
\begin{equation}\label{EQUATION5.2}
\left\{
	\begin{aligned}
	&S_{k_{i}}\left( D^{2}u_{i}\right) =|u_{i+1}|^{\gamma_{i}},  &{\rm in}\ \Omega,\\
	&u_i=0,  &{\rm on}\ \partial \Omega.
	\end{aligned}
\right.
\end{equation}
Notice that we denote $u_1:=u_{n+1}$ here.
It follows from Lemma \ref{Lemma5.2} that the admissible weak solution $\overline{T}_i(u_{i+1})\in C^0(\overline{\Omega})(i=1,\ldots,n)$. Define a composite operator $\overline{T}:=\overline{T}_1\overline{T}_2\cdots\overline{T}_n$, which is a completely continuous operator. Next, we verify that  $\overline{T}$ satisfies the assumptions of Lemma \ref{Lemma5.1}.

Due to the $k$-convexity property of the admissible weak solution and the boundary data, we have  $\overline{T}(X) \subseteq P$, which implies that $\overline{T}$ is positive and the operator $\overline{T}$ maps $P$ into itself. 
For $t>0$, we have
\begin{equation*}
	\overline{T}_1(tu_2)=t^{\frac{\gamma_1}{k_1}}\overline{T}_1(u_2),\quad \overline{T}_2(tu_3)=t^{\frac{\gamma_2}{k_2}}\overline{T}_2(u_3),\quad \cdots,\quad \overline{T}_n(tu_1)=t^{\frac{\gamma_n}{k_n}}\overline{T}_n(u_1).
\end{equation*} 
Since the assumption $\prod_{i=1}^{n}\gamma_i=\prod_{i=1}^{n}k_i$, we deduce that 
\begin{equation*}
 \overline{T}(tu_1)=t^{\frac{\prod_{i=1}^{n}\gamma_i}{\prod_{i=1}^{n}k_i}}\overline{T}(u_1)=t\overline{T}(u_1),	
\end{equation*}
which implies that $\overline{T}$ is homogeneous. By comparison principle in Lemma 2.1 in \cite{T} and the definition of $\overline{T}_i$, we get that $\overline{T}_i\ (i=1,\ldots,n)$ are all monotone, so is $\overline{T}$.
Finally, we just have to verify that  $\overline{T}$ is strong, that is, for all $u,v \in Im(\overline{T}) \cap P \setminus \{ \theta\} $, there exist $\delta>0$ and $\gamma>0$ such that $u-\delta v \le 0$ in $\Omega$ and $v-\gamma u \le 0$ in $\Omega$. 
If $u\in Im(\overline{T}) \cap P \setminus \{ \theta\} $, then there exists a function $v\in X\setminus \{ \theta\}$ such that $u=\overline{T}v=\overline{T}_1\overline{T}_2\cdots\overline{T}_n(v)$, where $u$ is nonzero $k$-admissible and strictly negative in $\Omega$ satisfying
\begin{equation*}
	\left\{
	\begin{aligned}
		&S_{k_{1}}\left( D^{2}u \right) =|v|^{\gamma{_1}},  &{\rm in}\ \Omega,\\
		&u=0,  &{\rm on}\ \partial \Omega.
	\end{aligned}
	\right.
\end{equation*}
It follows from Lemma \ref{Lemma5.2} that $u=\overline{T}v\in C^{0}(\overline{\Omega})$. 
Notice that $v$ is also the solution of \eqref{EQUATION5.2} satisfying $v \in C^{0}(\overline{\Omega})$, we let $v$ attains its minimum at $x_0 \in \overline{\Omega}$ and $G:= \max|v|^{\gamma_1}=(-v(x_0))^{\gamma_1}$, then we have $0< |v|^{\gamma_1}\le G$ in $\Omega$. 
Consider a function $\omega$ which satisfies
\begin{equation*}
	\left\{
	\begin{aligned}
		&S_{k_{1}}\left( D^{2}\omega \right) =G,  &{\rm in}\ \Omega,\\
		&\omega=0,  &{\rm on}\ \partial \Omega.
	\end{aligned}
	\right.
\end{equation*}
Then, it follows from Lemma \ref{Lemma5.4} that  $\omega \in C^{2,\alpha}(\overline{\Omega})$. 
By comparison principle in Lemma 2.1 in \cite{T}, we have $\omega\le u \le0$, in $\overline{\Omega}$ and $\omega= u =0$, on $\partial\Omega$.
Thus, for some small $t>0$, we have 
\begin{equation*}
	0\le \frac{u(x-t\nu)-u(x)}{-t}
	\le \frac{\omega(x-t\nu)-\omega(x)}{-t},\quad {\rm for}\ x\in\partial\Omega,
\end{equation*}
where $\nu$ is the unit outer normal vector field on $\partial\Omega$.
Take a limit in the last inequality, we have 
\begin{equation*}
	\begin{aligned}
		0\le \limsup_{t \to 0^{+}}\frac{u(x-t\nu)-u(x)}{-t}
		\le& \limsup_{t \to 0^{+}}\frac{\omega(x-t\nu)-\omega(x)}{-t}\\
		=&\frac{\partial\omega(x)}{\partial\nu},\quad {\rm for}\ x\in\partial\Omega.
	\end{aligned}
	 \end{equation*}
With the same argument, there also exists $\tilde{\omega}\in C^{2,\alpha}(\overline{\Omega})$ such that 
\begin{equation}\label{Equation5.2}
	\begin{aligned}
		0\le \limsup_{t \to 0^{+}}\frac{v(x-t\nu)-v(x)}{-t}
		\le& \limsup_{t \to 0^{+}}\frac{\tilde{\omega}(x-t\nu)-\tilde{\omega}(x)}{-t}\\
		=&\frac{\partial\tilde{\omega}(x)}{\partial\nu},\quad {\rm for}\ x\in\partial\Omega,
	\end{aligned}
\end{equation}
then by Hopf Lemma in \cite{Hopf}, we have
\begin{equation}\label{Equation5.3}
	\liminf_{t\to 0^{+}} \frac{u(x-t\nu)-u(x)}{-t}>0,  \quad {\rm on} \ \partial\Omega.
\end{equation} 
By choosing a sufficiently small constant $\delta_1>0$, we have
\begin{equation*}
	\begin{aligned}
		&\liminf_{t\to 0^{+}} \frac{(u-\delta_1v)(x-t\nu)-(u-\delta_1v)(x)}{-t}\\
		=&\liminf_{t\to 0^{+}} \frac{u(x-t\nu)-u(x)}{-t}- \delta_1 \limsup_{t \to 0^{+}} \frac{v(x-t\nu)-v(x)}{-t}\\
		>&0,\quad  {\rm on} \ \partial \Omega,
	\end{aligned}
    \end{equation*}
where \eqref{Equation5.2} and \eqref{Equation5.3} are used in the last inequality. Since $u,v$ are the solutions of \eqref{EQUATION5.2} which satisfy $u=v=0$ on $\partial \Omega$.
Then for $x\in \partial\Omega$, there exists a constant $t_0>0$ such that 
\begin{equation}\label{Equation5.11}
	(u-\delta_1v)(x-t\nu)<0, \quad {\rm for}\ t<t_0.
\end{equation}
Now, \eqref{Equation5.11} implies that 
\begin{equation*}
	u-\delta_1v<0, \quad {\rm in}\ \Omega_{t_0}:=\{x\in \Omega | {\rm dist}(x,\partial \Omega)< t_0\},
\end{equation*}
where $dist(x,\partial\Omega)$ denotes the distance from $x$ to $\partial\Omega$.
For $x\in\Omega \setminus \Omega_{t_0}$, we set 
\begin{equation*}
	\delta_2:=\inf_{x\in \Omega \setminus \Omega_{t_0}}\frac{u(x)}{v(x)}>0.
\end{equation*}
Fixing a constant $\delta\le \min\{\delta_1,\delta_2\}$, then we have 
\begin{equation*}
	u-\delta v\le u-\delta_1 v<0, \quad {\rm in} \ \Omega_{t_0} \quad {\rm and} \quad u-\delta v\le u-\delta_2 v\le 0, \quad {\rm in} \ \Omega\setminus \Omega_{t_0},
\end{equation*}
which implies 
\begin{equation*}
	u-\delta v \le 0, \quad {\rm in} \ \Omega.
\end{equation*}
The same argument shows that there exists a constant $\gamma$ such that $v-\gamma u \le 0$ in $\Omega$. Now we have shown that $\overline{T}$ is strong.
Moreover, $\mathcal{N}(\overline T)  := \{u\in P | \overline T (u)=0\}=\{0\}$.

Combining this with Lemma \ref{Lemma5.1} (i), we obtain that there exists 
$u_1^*\in P\setminus \{ \theta\}$ and constant $\lambda_0>0$ such that $u_1^*=\lambda_0\overline{T}(u_1^*)=\lambda_0\overline{T}_1\overline{T}_2\cdots\overline{T}_n(u_1^*)$. Let $u_n^*=\overline{T}_n(u_1^*),\ldots,u_2^*=\overline{T}_2(u_3^*)$. Then $(u_1^*,u_2^*,\ldots,u_n^*)$ is a solution of the following system
\begin{equation*}
	\left\{
	\begin{aligned}
		&S_{k_{1}}\left( D^{2}(\frac{u_{1}}{\lambda_0})\right) =\left(-u_{2}\right)^{\gamma_1} ,& {\rm in}\ \Omega,\\
		&S_{k_{2}}\left( D^{2}u_{2}\right) =\left(-u_{3}\right)^{\gamma_2} ,& {\rm in}\ \Omega,\\
		&\qquad \qquad \quad \vdots\\
		&S_{k_{n-1}}\left( D^{2}u_{n-1}\right) =\left(-u_{n}\right)^{\gamma_{n-1}} ,& {\rm in}\ \Omega,\\
		&S_{k_{n}}\left( D^{2}u_{n}\right) =\left(-u_{1}\right)^{\gamma_n} ,& {\rm in}\ \Omega,\\
		&u_{i}=0,i=1,\ldots,n,& {\rm on}\ \partial \Omega.
	\end{aligned}
	\right.
\end{equation*}
By Lemma \ref{Lemma5.1} (ii), if there exist $u_0 \in P \setminus \{ \theta\}$ and $\lambda_1 > 0$ such that $u_0 = \lambda_1 \overline{T}(u_0)$, then $\lambda_1 = \lambda_0$.

For this reason, the eigenvalue problem
\begin{equation}\label{Equation5.9}
	\left\{
	\begin{aligned}
		&S_{k_{1}}\left( D^{2}u_{1}\right) =\tilde{\lambda}\left(-u_{2}\right)^{\gamma_1} ,& {\rm in}\ \Omega,\\
		&S_{k_{2}}\left( D^{2}u_{2}\right) =\left(-u_{3}\right)^{\gamma_2} ,& {\rm in}\ \Omega,\\
		&\qquad \qquad \quad \vdots\\
		&S_{k_{n-1}}\left( D^{2}u_{n-1}\right) =\left(-u_{n}\right)^{\gamma_{n-1}} ,& {\rm in}\ \Omega,\\
		&S_{k_{n}}\left( D^{2}u_{n}\right) =\left(-u_{1}\right)^{\gamma_n} ,& {\rm in}\ \Omega,\\
		&u_{i}=0,i=1,\ldots,n,& {\rm on}\ \partial \Omega,
	\end{aligned}
	\right.
\end{equation}
admits a solution ($\boldsymbol{k}$-admissible solution) if and only if $\tilde{\lambda}=\lambda_0^{k_1}$.

Next, we prove that the system \eqref{EQUATION3} has a nonzero $\boldsymbol{k}$-admissible solution if and only if
\begin{equation*}
\lambda_1\lambda_2^{\frac{\gamma_1}{k_2}}\cdots\lambda_n^{\frac{\prod_{i=1}^{n-1}\gamma_i}{\prod_{i=2}^{n}k_i}}=\lambda_0^{k_1}.
\end{equation*}
In fact, if $(u_1,\ldots,u_n)$ is a  $\boldsymbol{k}$-admissible solution of the system \eqref{EQUATION3}, then 
\begin{equation*}
	S_{k_n}(D^2u_n)=\lambda_n(-u_1)^{\gamma_n},
\end{equation*}
which implies that
\begin{equation*}
	S_{k_n}\left( D^2(\lambda_n^{-\frac{1}{k_n}}u_n)\right) =(-u_1)^{\gamma_n},
\end{equation*}
Let $\tilde{u}_n=\lambda_n^{-\frac{1}{k_n}}u_n$, we have $	S_{k_n}( D^2\tilde{u}_n) =(-u_1)^{\gamma_n}$ and
\begin{equation*}
	S_{k_{n-1}}(D^2u_{n-1})=\lambda_{n-1}(-u_n)^{\gamma_{n-1}}=\lambda_{n-1}\lambda_n^{\frac{\gamma_{n-1}}{k_n}}(-\tilde{u}_n)^{\gamma_{n-1}}.
\end{equation*}
By the same argument, we let 
\begin{equation*}
	\begin{aligned}
		&\tilde{u}_{n-1}=\lambda_{n-1}^{-\frac{1}{k_{n-1}}}\lambda_n^{-\frac{\gamma_{n-1}}{k_{n-1}k_n}}u_{n-1},\\
		&\qquad\vdots\\
		&\tilde{u}_2=\lambda_{2}^{-\frac{1}{k_{2}}}\lambda_3^{-\frac{\gamma_2}{k_2k_{3}}}\cdots\lambda_n^{-\frac{\prod_{i=2}^{n-1}\gamma_i}{\prod_{i=2}^{n}k_i}}u_2,
	\end{aligned}
\end{equation*}  
therefore we have 
$S_{k_{n-1}}( D^2\tilde{u}_{n-1}) =(-\tilde{u}_{n})^{\gamma_{n-1}}$
,$\cdots$, $S_{k_1}(D^2u_1)=\lambda_1\lambda_{2}^{\frac{\gamma_1}{k_{2}}}
\cdots\lambda_n^{\frac{\prod_{i=1}^{n-1}\gamma_i}{\prod_{i=2}^{n}k_i}}(-\tilde{u}_{2})^{\gamma_1}$.
From the previous discussion, we know that \eqref{Equation5.9} admits a $\boldsymbol{k}$-admissible solution if and only if $\tilde{\lambda}=\lambda_0^{k_1}$. So,  $\lambda_1\lambda_2^{\frac{\gamma_1}{k_2}}\cdots\lambda_n^{\frac{\prod_{i=1}^{n-1}\gamma_i}{\prod_{i=2}^{n}k_i}}=\lambda_0^{k_1}$.

On the other hand, if  $\lambda_1\lambda_2^{\frac{\gamma_1}{k_2}}\cdots\lambda_n^{\frac{\prod_{i=1}^{n-1}\gamma_i}{\prod_{i=2}^{n}k_i}}=\lambda_0^{k_1}$, we let $\tilde{\lambda}=\lambda_1\lambda_2^{\frac{\gamma_1}{k_2}}\cdots\lambda_n^{\frac{\prod_{i=1}^{n-1}\gamma_i}{\prod_{i=2}^{n}k_i}}$. Then $\tilde{\lambda}=\lambda_0^{k_1}$, which implies that \eqref{Equation5.9} has a $\boldsymbol{k}$-admissible solution $(u_1,\ldots,u_n)$. Define
\begin{equation*}
	\begin{aligned}
	&u_2^*=\lambda_{2}^{\frac{1}{k_{2}}}\lambda_3^{\frac{\gamma_2}{k_2k_{3}}}\cdots\lambda_n^{\frac{\prod_{i=2}^{n-1}\gamma_i}{\prod_{i=2}^{n}k_i}}u_2,\\
	&\qquad\vdots\\
	&u_{n-1}^*=\lambda_{n-1}^{\frac{1}{k_{n-1}}}\lambda_n^{\frac{\gamma_{n-1}}{k_{n-1}k_n}}u_{n-1},\\
	&u_n^*=\lambda_n^{\frac{1}{k_n}}u_n.
	\end{aligned}
\end{equation*} 
Then $(u_1,u_2^*,\ldots,u_n^*)$ is a $\boldsymbol{k}$-admissible solution of system \eqref{EQUATION3}.
\end{proof}
\begin{remark}
	It is necessary to emphasis that if we define different composite operator, for example $\overline{T}:=\overline{T}_2\cdots\overline{T}_n\overline{T}_1$, etc, then we can let $\tilde{\lambda}$ be related to each $k_i\ (i=1,\ldots,n)$. Here, we only take $\tilde{\lambda}=\lambda_0^{k_1}$ for a detailed explanation.
\end{remark}
\begin{remark}
Here, we point out that our proof of the strong property of $\overline{T} $ is different from that in \cite{ZQ}. We overcome the difficult caused by the non-differentiability of solutions to degenerate $k$-Hessian equations and find the corresponding sub-solutions equipped with the higher regularity. Thanks to the Hopf lemma in \cite{Hopf} which applied to the non-differentiable function, we derive the proof.
\end{remark}

\vspace{3mm}

\bibliographystyle{amsplain}

\end{document}